\documentclass[11pt,letterpaper]{article}
\usepackage{fullpage}

\usepackage{amsfonts}
\usepackage{amsmath}
\usepackage{amsthm}
\usepackage{amssymb}

\usepackage[dvips]{graphicx}
\usepackage{color}
\usepackage{algorithmic}
\usepackage{algorithm}
\usepackage{multirow}

\theoremstyle{plain}
\newtheorem{theorem}{Theorem}

\newtheorem{lemma}[theorem]{Lemma}

\theoremstyle{definition}

\newtheorem{assumption}[theorem]{Assumption}

\newtheorem{remark}[theorem]{Remark}

\newcommand{\p}{\mathcal{P}_i}
\newcommand{\hp}{\hat{\mathcal{P}}_i}
\newcommand{\R}{\mathbb{R}}
\newcommand{\E}{\mathbb{E}}
\newcommand{\TT}{T_{\delta}}
\newcommand{\Ai}{A_i}
\newcommand{\Ci}{C_i}
\newcommand{\Di}{D_i}

\newcommand{\eqdef}{:=}
\newcommand{\ii}{^{(i)}}

\newcommand{\Lip}{l}
\newcommand{\probs}{p}
\DeclareMathOperator{\dom}{dom}
\newcommand{\Prob}{\mathbb{P}}

\title{Inexact Coordinate Descent: Complexity and Preconditioning\footnote{This work was supported by the EPSRC grant EP/I017127/1 ``Mathematics for vast digital resources''. Peter Richt\'{a}rik  was also supported by the Centre for Numerical Algorithms and Intelligent Software (funded by EPSRC grant EP/G036136/1 and the Scottish Funding Council).}}

\author{Rachael Tappenden \qquad   Peter Richt\'{a}rik \qquad   Jacek Gondzio\\\\\emph{School of Mathematics}\\\emph{University of Edinburgh}\\\emph{United Kingdom}}

\begin{document}

\maketitle

\begin{abstract}
In this paper we consider the problem of minimizing a convex function using a randomized block coordinate descent method. One of the key steps at each iteration of the algorithm is determining the update to a block of variables. Existing algorithms assume that in order to compute the update, a particular subproblem is solved \emph{exactly}. In this work we relax this requirement, and allow for the subproblem to be solved \emph{inexactly}, leading to an inexact block coordinate descent method. Our approach incorporates the best known results for exact updates  as a special case. Moreover, these theoretical guarantees are complemented by practical considerations: the use of iterative techniques to determine the update as well as the use of preconditioning for further acceleration.

\bigskip
\textbf{Keywords:} inexact methods,  block coordinate descent,  convex optimization, iteration complexity, preconditioning, conjugate gradients.

\bigskip
\textbf{AMS:} 65F08; 65F10; 65F15; 65Y20; 68Q25; 90C25
\end{abstract}

\section{Introduction}
Due to a dramatic increase in the size of optimization problems being encountered, first order  methods are becoming increasingly popular. These large-scale problems are often highly structured and it is important for any optimization method to take advantage of the underlying structure. Applications where such problems arise and where first order methods have proved successful include machine learning \cite{Machart12,Schmidt11}, compressive sensing \cite{Donoho06,Wright09}, group lasso \cite{Qin10,Simon11}, matrix completion \cite{Candes12,Recht11}, and truss topology design \cite{Richtarik11a}.

Block coordinate descent methods seem a natural choice for these very large-scale problems due to  their low memory requirements and low per-iteration computational cost. Furthermore, they are often designed to take advantage of the underlying structure of the optimization problem \cite{Tseng09,Wright12} and many of these algorithms are supported by high probability iteration complexity results \cite{Nesterov07,Nesterov12,Richtarik11a,Richtarik12,Richtarik13}.

\subsection{The Problem}
If the block size is larger than one, determining the update to use at a particular iteration in a block coordinate descent method can be computationally expensive. The purpose of this work is to reduce the cost of this step. To achieve this, we extend the work in \cite{Richtarik12} to include the case of an \emph{inexact} update.

In this work we study randomized block coordinate descent methods applied to the problem of minimizing a composite objective function. That is, a function formed as the sum of a smooth convex and a simple nonsmooth convex term:
\begin{equation}
     \label{F}
\min_{x \in \mathbb{R}^N} \{F(x) \eqdef f(x) + \Psi (x)\}.
\end{equation}
We assume that the problem has a minimum $(F^* > -\infty)$, $f$ has (block) coordinate Lipschitz gradient, and $\Psi$ is a (block) separable proper closed convex extended real valued function (all these concepts will be defined precisely in Section~\ref{Section_Preliminaries}).

Our algorithm (namely, the Inexact Coordinate Descent (ICD) method) is supported by high probability iteration complexity results. That is, for confidence level $\rho \in (0,1)$ and error tolerance $\epsilon >0$, we give an explicit expression for the number of iterations $k$ that guarantee that the ICD method produces a random iterate $x_k$ for which
\begin{equation*}
     \Prob(F(x_k)-F^* \leq \epsilon ) \geq 1-\rho.
\end{equation*}
We will show that in the inexact case it is not always possible to achieve a solution with small error and/or high confidence.

Our theoretical guarantees are complemented by practical considerations. In Section \ref{Section_alphabeta} we explain our inexactness condition in detail and in Section \ref{S_VerifyVi} we give examples to show when the inexactness condition is implementable. Further, in Section \ref{Section_Practical} we give several examples, derive the update subproblems, and suggest algorithms that could be used to solve the subproblems inexactly. Finally, we present some encouraging computational results.

\subsection{Literature Review}
As problem sizes increase, first order methods are benefiting from revived interest. On very large problems however, the computation of a single gradient step is expensive, and methods are needed that are able to make progress before a standard gradient algorithm takes a single step. For instance, a randomized variant of the Kaczmarz method for solving linear systems has recently been studied, equipped with iteration complexity bounds \cite{Needell10,Needell12,Leventhal10,Strohmer09}, and found to be surprisingly efficient.
This method can be seen as a special case of a more general class of decomposition algorithms, block coordinate descent methods, which have recently gained much popularity \cite{Necoara12a,Nesterov07,Nesterov12,Richtarik12,Richtarik12a,Richtarik11,Tseng01}. One of the main differences between various (serial) coordinate descent schemes is the way in which the coordinate is chosen at each iteration. Traditionally cyclic schemes \cite{Saha10} and greedy schemes \cite{Richtarik11a} were studied. More recently, a popular alternative is to select coordinates randomly, because the coordinate can be selected cheaply, and useful iteration complexity results can be obtained \cite{Necoara12,Richtarik12,Richtarik12a,Richtarik11,Tao12,SSS2013}.

Another current trend in this area is to consider methods that incorporate some kind of `inexactness', perhaps using approximate gradients, or using inexact updates. For example, \cite{Necoara12b} considers methods based on inexact dual gradient information, while \cite{Schmidt11} considers the minimization of an unconstrained convex composite function where error is present in the gradient of the smooth term, or in the proximity operator for the non-smooth term. Other works study methods that use inexact updates when the objective function is convex, smooth and unconstrained \cite{Bento12}, smooth and constrained \cite{Bonettini11} or for $\ell_1$-regularized quadratic least squares problem \cite{Hua12}.

\subsection{Contribution} \label{Section_Comparison}

In this paper we extend the work of Richt\'{a}rik and Tak\'a\v{c} \cite{Richtarik12} and present a block coordinate descent method that employs inexact updates having the potential to reduce the overall algorithm running time. Furthermore, we focus in detail on the quadratic case, which benefits greatly from inexact updates, and show how preconditioning can be used to complement the inexact update strategy.

\begin{table}[h!]\centering
{
\begin{tabular}{|c| c| c| c|}
\hline
$F$ & Exact Method \cite{Richtarik12} & Inexact Method [this paper] & Theorem \\
\hline
& & & \\[-3mm]
C-N & $\displaystyle\frac{c_1}{\epsilon}\left(1+\log{\frac{1}{\rho}}\right) + 2$
& $\displaystyle\frac{c_1}{\epsilon-u} + \frac{c_1}{\epsilon  - \alpha c_1} \log \left( \frac{\epsilon - \frac{\beta c_1}{\epsilon - \alpha c_1}}{\epsilon\rho - \frac{\beta c_1}{\epsilon - \alpha c_1}}\right) + 2$ & \ref{Thm:C-N}(i)\\[5mm]
\cline{2-4}
& & & \\[-3mm]
C-N & $c_2 \log{\left(\displaystyle\frac{F(x_0)-F^*}{\epsilon \rho}\right)}$& $\displaystyle\frac{c_2}{1-\alpha c_2} \log \left(\frac{F(x_0)-F^* - \frac{\beta c_2}{1-\alpha c_2}}{\epsilon \rho - \frac{\beta c_2}{1-\alpha c_2}}\right)$ & \ref{Thm:C-N}(ii)\\[5mm]
\hline
& & & \\[-3mm]
SC-N & $\displaystyle\frac{n}{\mu} \log{\left(\frac{F(x_0) - F^*}{\epsilon \rho} \right)}$ & $ \displaystyle\frac{n}{\mu-\alpha n} \log{\left(\frac{F(x_0) - F^* - \frac{\beta n}{\mu-\alpha n}}{\epsilon \rho - \frac{\beta n}{\mu-\alpha n}} \right)}$ & \ref{Thm_ICD_strong} \\[5mm]
\hline
& & & \\[-3mm]
C-S & $\displaystyle\frac{\hat{c}_1}{\epsilon}\left(1+\log{\frac{1}{\rho}}\right) + 2$ & $\displaystyle\frac{\hat{c}_1}{\epsilon - \hat{u}} + \frac{\hat{c}_1}{\epsilon  - \alpha \hat{c}_1} \log \left( \frac{\epsilon - \frac{\beta \hat{c}_1}{\epsilon - \alpha \hat{c}_1}}{\epsilon\rho - \frac{\beta \hat{c}_1}{\epsilon - \alpha \hat{c}_1}}\right) + 2$ & \ref{Thm:C-S} \\[5mm]
\hline
& & & \\[-3mm]
SC-S & $\displaystyle\frac{1}{\mu_f}\log\left(\frac{f(x_0)-f^*}{\epsilon \rho}\right)$ &  \par$\displaystyle\frac{1}{\mu_f - \alpha}\log\left(\frac{f(x_0)-f^* - \frac{\beta}{\mu_f - \alpha}}{\epsilon \rho- \frac{\beta}{\mu_f -\alpha}}\right)$ & \ref{Thm:SC-S}\\[5mm]
\hline
\end{tabular}
}
\caption{Comparison of the iteration complexity results for coordinate descent methods using an inexact update and using an exact update (C=Convex, SC=Strongly Convex, N=Nonsmooth, S = Smooth).}\vspace{-5mm}
\label{Table_Comparison}
\end{table}

Table~\ref{Table_Comparison} compares \emph{some} of the new complexity results obtained in this paper for an inexact update with the complexity results for an exact update presented in \cite{Richtarik12}. The following notation is used in the table: by $\mu_\phi$ we denote the strong convexity parameter of function $\phi$ (with respect to a certain norm specified later), $\mu = (\mu_f + \mu_\Psi)/(1 + \mu_\Psi)$ and ${\cal R}_{w}(x_0)$ can be roughly considered to be distance from $x_0$ to a solution of \eqref{F} measured in a specific weighted norm parameterized by the vector $w$ (to be defined precisely in \eqref{eq_R}). The constants are $c_1 = 2n \max\{{\cal R}_{w}^2(x_0),F(x_0)-F^*\}, \hat{c}_1 = 2{\cal R}_{w}^2(x_0)$ and $c_2 = 2n{\cal R}_{w}^2(x_0)/\epsilon$, and $n$ is the number of blocks. Parameters $\alpha, \beta \geq 0$ control the level of inexactness (to be defined precisely in Section \ref{S_delta}) and $u$ and $\hat{u}$ are constants depending on $\alpha$, $\beta$ and $c_1$, and $\alpha$, $\beta$ and $\hat{c}_1$, respectively.

Table~\ref{Table_Comparison} shows that for fixed $\epsilon$ and $\rho$, an inexact method will require more iterations than an exact one. However, it is expected that in certain situations an inexact update will be significantly cheaper to compute than an exact update, leading to better overall running time.
Moreover, the new complexity results for the inexact method generalize those for the exact method. Specifically, for  inexactness parameters  $\alpha = \beta = 0$ we recover the complexity results in \cite{Richtarik12}.

\subsection{Outline} The first part of this paper focuses on the theoretical aspects of a block coordinate descent method when an inexact update is employed. In Section~\ref{Section_Preliminaries} the assumptions and notation are laid out and in Section~\ref{Section_Algorithms} the ICD method is presented. In Section~\ref{Section_CC} iteration complexity results for ICD applied to \eqref{F} are presented in both the convex and strongly convex cases. Iteration complexity results for ICD applied to a convex smooth minimization problem ($\Psi = 0$ in \eqref{F}) are presented in Section~\ref{Section_CS}, in both the convex and strongly convex cases.

The second part of the paper considers the practicality of an inexact update. Section~\ref{Section_Practical} provides several examples of how to derive the formulation for the update step subproblem, as well as giving suggestions for algorithms that can be used to solve the subproblem inexactly. Numerical experiments are presented in Section~\ref{Section_Numerical} and Appendix~\ref{Section_Eigenvalues} provides a detailed analysis of the spectrum of the preconditioned matrix used in the numerical experiments.

\section{Assumptions and Notation}
\label{Section_Preliminaries}
In this section we introduce the notation and definitions that are used throughout the paper.

\subsection{Block structure of $\R^N$}
\label{S_Block_structure}

The problem under consideration is assumed to have block structure and this is modelled by decomposing the space $\mathbb{R}^N$ into $n$ subspaces as follows. Let $U \in \mathbb{R}^{N \times  N}$ be a column permutation of the $N \times N$ identity matrix and further let $U = [U_1,U_2,\dots,U_n]$ be a decomposition of $U$ into $n$ submatrices, where $U_i$ is $N \times N_i$ and $\sum_{i=1}^n N_i = N$. It is clear (e.g., see \cite{Richtarik12a} for a brief proof) that any vector $x \in \mathbb{R}^N$ can be written uniquely as
\begin{equation}\label{eq:decomposition}x = \sum_{i=1}^n U_ix^{(i)},\end{equation} where $x^{(i)} \in \mathbb{R}^{N_i}$. Moreover, these vectors are given by
\begin{equation} \label{U_i} x^{(i)} \eqdef U_i^Tx.\end{equation}
For simplicity we will sometimes write $x = (x^{(1)},x^{(2)},\dots,x^{(n)})$ instead of \eqref{eq:decomposition}. We equip $\mathbb{R}^{N_i}$ with a pair of conjugate norms, induced by a quadratic form involving a symmetric positive definite matrix $B_i$:
\begin{eqnarray}
\label{S2_Norm_def}
  \|t\|_{(i)} \eqdef \langle B_it,t \rangle^{\frac{1}{2}}, \qquad \|t\|_{(i)}^* = \langle B_i^{-1}t,t \rangle^{\frac{1}{2}}, \qquad t \in \mathbb{R}^{N_i},
\end{eqnarray}
where $\langle \cdot,\cdot \rangle$ is the standard Euclidean dot product.

\subsection{Smoothness of $f$}
Throughout this paper we assume that the gradient of $f$ is block  Lipschitz, uniformly in $x$, with positive constants $\Lip_1,\dots,\Lip_n$. This means that, for all $x \in \R^N$, $i\in \{ 1,2,\dots,n\}$ and $t \in \R^{N_i}$ we have
\begin{equation}
\label{S2_Lipschitz}
     \| \nabla_i f(x + U_i t) - \nabla_i f(x) \|_{(i)}^* \leq \Lip_i \|t\|_{(i)},
\end{equation}
where
\begin{equation}
\label{block_gradient}
     \nabla_i f(x) \eqdef (\nabla f(x))^{(i)} \overset{\eqref{U_i}}{=} U_i^T\nabla f(x) \in \R^{N_i}.
\end{equation}
An important consequence of \eqref{S2_Lipschitz} is the following standard inequality \cite[p.57]{Nesterov04}:
\begin{equation}
\label{S2_upperbound}
     f(x+U_i t) \leq f(x) + \langle \nabla_i f(x), t\rangle+ \tfrac{\Lip_i}{2}\|t\|_{(i)}^2.
\end{equation}

\subsection{Block separability of $\Psi$}
The function $\Psi:\R^N \to \R\cup \{+\infty\}$ is assumed to be block separable. That is, we assume that it can be decomposed as:
\begin{equation}
     \label{S2_separable_psi}
\Psi(x) = \sum_{i=1}^n \Psi_i (x^{(i)}),
\end{equation}
where the functions $\Psi_i: \R^{N_i} \to \R\cup \{+\infty\}$ are convex and closed.

\subsection{Norms on $\R^N$}

For fixed positive scalars $w_1,w_2,\dots,w_n$, let $w = (w_1,\dots,w_n)$ and define a pair of conjugate norms in $\R^N$ by
\begin{equation}
\label{S_Norms_1}
     \|x\|_w^2 \eqdef  \sum_{i=1}^n w_i \|x^{(i)}\|^2_{(i)}, \quad   (\|y\|_w^*)^2 \eqdef \max_{\|x\|_w \leq 1} \langle y,x \rangle ^2 = \sum_{i=1}^n w_i^{-1} (\|y^{(i)}\|^*_{(i)})^2.
\end{equation}
In the subsequent analysis we will often use $w = \Lip$ (for $\Psi\neq 0$) and/or $w = \Lip  p^{-1}$ (for $\Psi=0$), where $\Lip = (\Lip_1,\dots,\Lip_n)$ is a vector of Lipschitz constants, $p = (p_1,\dots,p_n)$ is a vector of positive probabilities and $\Lip p^{-1}$ denotes the vector $(\Lip_1/p_1,\dots, \Lip_n/p_n)$.

\subsection{Strong convexity of $F$}
\label{S_StrongConvex}
A function $\phi: \R^N \to \R \cup \{+ \infty\}$ is strongly convex w.r.t.\ $\| \cdot \|_w$ with convexity parameter $\mu_{\phi}(w) > 0$ if for all $x,y \in \dom \phi$,
\begin{equation}
\label{strongly_convex_1}
     \phi(y) \geq \phi(x) + \langle \phi^{\prime}(x),y-x \rangle + \tfrac{\mu_{\phi}(w)}{2}\|y-x\|_w^2,
\end{equation}
where $\phi^{\prime}$ is any subgradient of $\phi$ at $x$. The case with $\mu_{\phi}(w)= 0$ reduces to convexity.

In some of the results presented in this work we assume that $F$ is strongly convex. Strong convexity of $F$ may come from $f$ or $\Psi$ or both and we will write $\mu_f(w)$ (resp. $\mu_{\Psi}(w)$) for the strong convexity parameter of $f$ (resp. $\Psi$), with respect to $\| \cdot \|_w$. Following from \eqref{strongly_convex_1}
\begin{equation}
     \label{strongly_convex_4}
\mu_F(w) \geq \mu_f(w) + \mu_\Psi(w).
\end{equation}
Using \eqref{S2_upperbound} and \eqref{strongly_convex_1} it can be shown that
\begin{equation}
\label{eq_lipp1}
	\mu_f(\Lip)\leq 1, \quad \text{and} \quad \mu_f(\Lip  p^{-1}) < 1.
\end{equation}
We will also make use of the following characterisation of strong convexity. For all $x,y \in \dom\phi$ and $\lambda \in [0,1]$,
\begin{equation}
\label{strongly_convex_3}
     \phi\big(\lambda x + (1-\lambda) y \big) \; \leq \; \lambda \phi(x) + (1-\lambda)\phi(y) - \tfrac{\mu_\phi(w)\lambda(1-\lambda)}{2}\|x-y\|_w^2.
\end{equation}

\subsection{Level set radius}
The set of optimal solutions of \eqref{F} is denoted by $X^*$ and $x^*$ is any element of that set. We define
\begin{equation}
\label{eq_R}
     \mathcal{R}_w(x) \eqdef \max_y \max_{x^* \in X^*} \{ \|y - x^*\|_w : F(y) \leq F(x) \},
\end{equation}
which is a measure of the size of the level set of $F$ given by $x$. We assume that $\mathcal{R}_w(x_0)$ is finite for the initial iterate $x_0$.

\section{The Algorithm}
\label{Section_Algorithms}
Let us start by presenting the algorithm; a more detailed description will follow.
\begin{algorithm}[H]
\caption{ICD: Inexact Coordinate Descent}\label{ICD}
\begin{algorithmic}[1]
\STATE Input: Inexactness parameters $\alpha,\beta \geq 0$, and probabilities $p_1,\dots,p_n>0$.
\FOR{$k = 0,1,2,\dots$}
\STATE Choose $\delta_k = (\delta_k^{(1)},\dots,\delta_k^{(n)}) \in \R^n$ according to \eqref{eq:alpha-beta}
\STATE Choose block $i \in \{1,2,\dots,n\}$ with probability $p_i$
\STATE Compute the inexact update $T^{(i)}_{\delta_k}(x_k)$ to block $i$ of $x_k$
\STATE Update block $i$ of $x_k$: $x_{k+1} = x_k + U_i T^{(i)}_{\delta_k}(x_k)$
\ENDFOR
\end{algorithmic}
\end{algorithm}

\subsection{Generic description}

Given iterate $x_k \in \R^N$, Algorithm~\ref{ICD} picks block $i \in \{1,2,\dots,n\}$ with probability $p_i$, computes the update vector $T^{(i)}_{\delta_k}(x_k)\in \R^{N_i}$ (we comment on how this is computed later in this section) and then adds it to the $i$th block of $x_k$, producing the new iterate $x_{k+1}$. The iterates $\{x_k\}$ are random vectors and the values $\{F(x_k)\}$ are random variables. The update vector depends on $x_k$, the current iterate, and on $\delta_k$, a vector of parameters controlling the ``level of inexactness'' with which the update is computed. The rest of this section is devoted to giving a precise definition of $T^{(i)}_{\delta_k}(x_k)$. Note that from \eqref{F} and \eqref{S2_upperbound} we have, for all $x\in \R^N$, $i \in \{1,2,\dots,n\}$ and $t\in \R^{N_i}$:
\begin{equation}
\label{S2_upperbound_F}
     F(x + U_it) = f(x+U_it) + \Psi(x+U_i t) \leq f(x) + V_i(x,t) +  \Psi_{-i}(x),
\end{equation}
where
\begin{equation}
     \label{Vi}
  V_i(x,t) \eqdef \langle \nabla_i f(x),t \rangle + \tfrac{\Lip_i}{2}\|t\|_{(i)}^2 + \Psi_i(x^{(i)}+t),
\end{equation}
\begin{equation}
\label{psi}
     \Psi_{-i}(x) \eqdef \sum_{j \neq i} \Psi_j(x^{(j)}).
\end{equation}
That is, \eqref{S2_upperbound_F} gives an upper bound on $F(x+U_i t)$, viewed as a function of $t \in\R^{N_i}$.

The inexact update computed in Step 5 of Algorithm \ref{ICD} is the \emph{inexact} minimizer of the upper bound \eqref{S2_upperbound_F} on $F(x_k+U_it)$ (to be defined precisely below). However, since only the second term of this bound depends on $t$, the update is computed by minimizing, \emph{inexactly}, $V_i(x,t)$ in $t$.

\subsection{Inexact update}
\label{S_delta}
The approach of this paper best applies to situations in which it is much easier to approximately minimize $t\mapsto V_i(x,t)$ than to either
(i) approximately minimize $t\mapsto F(x+ U_i t)$ and/or (ii) exactly minimize $t \mapsto V_i(x,t)$.
For $x\in \R^N$ and $\delta = (\delta^{(1)},\dots, \delta^{(n)}) \geq 0$
we define $T_{\delta}(x) \eqdef (T_{\delta}^{(1)}(x),\dots, T_{\delta}^{(n)}(x))\in \R^N$ to be any vector satisfying
\begin{equation}
\label{Td_def}
     V_i(x,T_{\delta}^{(i)}(x)) \leq  \min \left\{V_i(x,0), \delta^{(i)}+ \min_{t \in \R^{N_i}} V_i(x,t)\right\}, \quad  i = 1,\dots,n.
\end{equation}
(We allow here for an abuse of notation --- $\delta^{(i)}$ is a scalar, rather than a vector in $\R^{N_i}$ as $x^{(i)}$ for $x\in \R^N$ --- because we wish to  emphasize that the scalar $\delta^{(i)}$ is associated with the $i$th block.)
That is, we require that the inexact update $T_\delta^{(i)}(x)$ of the $i$th block of $x$ is (i) no worse than a vacuous update, and that it is (ii) close to the optimal update $T_0^{(i)}(x)=\arg \min_{t} V_i(x,t)$, where the degree of suboptimality/inexactness is bounded by $\delta^{(i)}$.

As the following lemma shows, the update \eqref{Td_def} leads to a monotonic algorithm.

\begin{lemma}  For all $x\in \R^N, \delta \in \R^n_+$ and $i\in \{1,2,\dots,n\}$,
\begin{equation}\label{eq:111}F(x + U_i T_{\delta}^{(i)}(x)) \leq F(x).\end{equation}
\end{lemma}
{\em Proof:}
\begin{eqnarray*}
F(x+ U_i T_{\delta}^{(i)}(x)) &\overset{\eqref{S2_upperbound_F}}{\leq}& f(x) + V_i(x, T_\delta^{(i)}(x)) + \Psi_{-i}(x)\\
&\overset{\eqref{Td_def}}{\leq}& f(x) + V_i(x, 0) + \Psi_{-i}(x) \overset{\eqref{Vi} + \eqref{psi}}{=}  F(x).
\qquad \qed
\end{eqnarray*}

Furthermore, in this work we provide iteration complexity results for ICD, where $\delta_k=(\delta_k^{(1)},\dots,\delta_k^{(n)})$ is chosen in such a way that the expected suboptimality is bounded above by a linear function of the residual $F(x_k)-F^*$. That is, we have the following assumption.
\begin{assumption}\label{assumption}
  For constants $\alpha,\beta \geq 0$, the vector $\delta_k= (\delta_k^{(1)},\dots, \delta_k^{(n)})$ is chosen to satisfy
  \begin{equation}\label{eq:alpha-beta}\bar{\delta}_k \eqdef \sum_{i=1}^n p_i \delta_k^{(i)} \leq \alpha (F(x_k)-F^*) + \beta,\end{equation}
\end{assumption}
Notice that, for instance, Assumption \ref{assumption} holds if we require $\delta_k^{(i)}\leq \alpha (F(x_k)-F^*)+ \beta$ for all blocks $i$ and iterations $k$.

The motivation for allowing inexact updates of the form \eqref{Td_def} is that calculating exact updates is impossible in some cases (for example, not all problems have a closed form solution), and computationally intractable in others. The purpose of allowing inexactness in the update step of ICD is that an \emph{iterative method} can be used to solve for the update $T_{\delta}^{(i)}(x)$, thus significantly expanding the range of problems that can be successfully tackled by coordinate descent. In this case, there is an outer coordinate descent loop, and an inner iterative loop to determine the update. Assumption \ref{assumption} shows that the stopping tolerance on the inner loop, must be bounded above via  \eqref{eq:alpha-beta}.

CD methods provide a mechanism to break up very large/huge scale problem, into smaller pieces that are a fraction of the total dimension. Moreover, often the subproblems that arise to solve for the update have a similar/the same form as the original huge scale problem. (For example, see the numerical experiments in Section \ref{Section_Numerical}, and the examples given in Section \ref{Section_Practical}.) There are many iterative methods that cannot scale up to the original huge dimensional problem, but are excellent at solving the medium scale update subproblems. ICD allows these algorithms to solve for the update at each iteration, and if the updates are solved efficiently, then the overall ICD algorithm running time is kept low.

\subsection{The role of $\alpha$ and $\beta$ in ICD}
\label{Section_alphabeta}
The condition \eqref{Td_def} shows that the updates in ICD are \emph{inexact}, while Assumption \ref{assumption} gives the \emph{level of inexactness} that is allowed in the computed update. Moreover, Assumption \ref{assumption} allows us to provide a unified analysis; formulating the error/inexactness expression in this general way \eqref{eq:alpha-beta} gives insight into the role of both \emph{multiplicative and additive error}, and \emph{how this error propagates} through the algorithm as iterates progress.

Formulation \eqref{eq:alpha-beta} is interesting from a theoretical perspective because it allows us to present a \emph{sensitivity analysis} for ICD, which is interesting in its own right. However, we stress that \eqref{Td_def}, coupled with Assumption \ref{assumption}, is much more than just a technical tool; $\alpha$ and $\beta$ are actually parameters of the ICD algorithm (Algorithm \ref{ICD}) that can be assigned \emph{explicit numerical values} in many cases.

We now explain \eqref{Td_def}, Assumption \ref{assumption} and the role of parameters $\alpha$ and $\beta$ in slightly more detail. (Note that $\alpha$ and $\beta$ must be chosen sufficiently small to guarantee converge of the ICD algorithm. However, we postpone discussion of the magnitude of $\alpha$ and $\beta$ until Section \ref{S_technicalresult}.) There are four cases.
\begin{enumerate}
\item \textbf{Case I: $\alpha = \beta = 0$.} This corresponds to the \emph{exact case} where \emph{no error} is allowed in the computed update.
  \item \textbf{Case II: $\alpha = 0, \beta >0$.} This case corresponds to additive error only, where the error level $\beta>0$ is fixed at the start of Algorithm \ref{ICD}. In this case, \eqref{Td_def} and \eqref{eq:alpha-beta} show that the error allowed in the inexact update $T_{\delta}^{(i)}(x_k)$ is \emph{on average $\beta$}. For example, one can set $\delta_k\ii = \beta$, for all blocks $i$ and all iterations $k$ so that \eqref{eq:alpha-beta} becomes $\bar{\delta}_k = \sum_i p_i \beta = \beta$. Notice that the tolerance allowable on each block need not be the same; if one sets $\delta_k\ii \leq \beta$, for all blocks $i$ and iterates $k$ then $\bar{\delta}_k \leq \beta$, so \eqref{eq:alpha-beta} holds true. Moreover, one need not set $\delta_k\ii >0$ for all $i$, so that the update vector $T_{\delta}^{(i)}(x_k)$ could be exact for some blocks ($\delta_k\ii =0$), and inexact for others ($\delta_k^{(j)} >0$). (This may be sensible, for example, when $\Psi_i(x\ii) \neq \Psi_j(x^{(j)})$ for some $i\neq j$ and that \eqref{Vi} has a closed form solution for $T\ii(x_k)$ but not for $T^{(j)}(x_k)$.) Furthermore, consider the extreme case where only one block update is inexact $T_{\delta}\ii(x_k)$, ($T_{0}^{(j)}(x_k)$ for all $j\neq i$). If the coordinates are selected with uniform probability, then the inexactness level on block $i$ can be as large as $\delta_k\ii = n \beta$ and Assumption \ref{assumption} holds.
  \item \textbf{Case III: $\alpha > 0, \beta = 0$.} In this case only multiplicative error is allowed in the computed update $T_{\delta}\ii(x_k)$, where the error allowed in the update at iteration $k$ is related to the error in the function value ($F(x_k)-F^*$). The multiplicative error level $\alpha$ is fixed at the start of Algorithm \ref{ICD}, and $\alpha(F(x_k) - F^*)$ is an upper bound on the average error in the update $T_{\delta}^{(i)}(x_k)$ over all blocks $i$ at iteration $k$. In particular, notice that setting $\delta_k\ii \leq \alpha (F(x_k)-F^*)$, for all $i$ and $k$ satisfies Assumption \ref{assumption}. As for Case II, one is allowed to set $\delta_k\ii = 0$ for some block(s) $i$, or to set $\delta_k\ii > \alpha (F(x_k)-F^*)$ for some blocks $i$ and iterations $k$ as long as Assumption \ref{assumption} is satisfied.
  \item \textbf{Case IV: $\alpha > 0, \beta > 0$.} This is the most general case, corresponding to the inclusion of both multiplicative and additive error, where the error level parameters $\alpha$ and $\beta$ are fixed at the start of Algorithm \ref{ICD}. Notice that Assumption \ref{assumption} is satisfied when the error $\delta_k\ii$ in the computed update $T_{\delta}^{(i)}(x_k)$ obeys $\delta_k\ii \leq \alpha(F(x_k)-F^*)+\beta$. Moreover, as for Cases II and III, one is allowed to set $\delta_k\ii = 0$ for some block(s) $i$, or to set $\delta_k\ii > \alpha (F(x_k)-F^*)$ for some blocks $i$ and iterations $k$ as long as Assumption \ref{assumption} is satisfied. Notice that, as iterations progress, the multiplicative error $\alpha$ may become dominated by the additive error $\beta$, in the sense that, $\alpha(F(x_k)-F^*) \to 0$ as $k \to \infty$ so the upper bound on $\bar{\delta_k}$ tends to $\beta$.
      \end{enumerate}

Cases I--IV above show that the parameters $\alpha$ and $\beta$ directly relate to the stopping criterion used in the algorithm employed to solve for the update $T_{\delta}^{(i)}(x_k)$ at each iteration of ICD. The following section gives examples of algorithms that can be used within ICD, where $\alpha$ and $\beta$ can be given explicit numerical values and the stopping tolerances are verifiable.

\subsection{Computing the inexact update}
\label{S_VerifyVi}
In this section we focus on the computation of the inexact update (Step 5 of Algorithm \ref{ICD}). We discuss several cases where it is possible to verify Assumption \ref{assumption}, and thus provide specific instances to show that ICD is indeed implementable.

In order to compute an inexact update, the quantity $V_i(x,T_0\ii(x_k))$ in \eqref{Td_def} is needed. Moreover, to incorporate multiplicative error, the optimal objective value $F^*$ must also be known.

\begin{enumerate}
\item \textbf{$F^*$ is known:} There are many instances when $F^*$ is known a priori, which means that the bound $F(x_k)-F^*$ is computable at every iteration, and subsequently multiplicative error can be incorporated into ICD. In most cases, the update subproblem \eqref{Td_def} has the same form as the original problem (see Section \ref{Section_Practical}), so that $V_i(x_k,T_0\ii)$ is also known. This is the case, for example, when solving a consistent system of equations (minimizing a quadratic function), where $F^*=0$, so that $F(x_k)-F^* = F(x_k)$. The update subproblem has the same form (a consistent system of equations must be solved to calculate $T_{\delta_k}\ii$), so we also have $V_i(x_k,T_0\ii) = 0$. Hence, for any $\alpha,\beta \geq 0$, one can compute $\delta_k\ii$ \eqref{eq:alpha-beta}, that satisfies Assumption \ref{assumption} and \eqref{Td_def}. That is, at each iteration of ICD, accept any inexact update $T_{\delta_k}\ii$ that satisfies:
    \begin{equation}
      \label{StoppingCondition_Fstar}
        V_i(x_k,T_{\delta_k}\ii) - V_i(x_k,T_0\ii) = V_i(x_k,T_{\delta_k}\ii) \leq \delta_k\ii \leq \alpha F(x_k) + \beta.
      \end{equation}
  \item \textbf{Primal-dual algorithm:} 
      If the update $T_{\delta_k}\ii$ is found using an algorithm that terminates on the duality gap, then \eqref{Td_def} and Assumption \ref{assumption} are easy to verify. In particular, suppose that $\alpha = 0$ and $\beta>0$ is fixed when initializing ICD. Then, for block $i$ and iteration $k$, we accept $T_{\delta_k}\ii$ such that
      \begin{equation}
      \label{StoppingCondition_Dual}
        V_i(x_k,T_{\delta_k}\ii)) - V_i(x_k,T_0\ii) \leq V_i(x_k,T_{\delta_k}\ii)) - V_i^{\text{DUAL}}(x_k,T_{\delta_k}\ii))\leq \delta_k\ii \leq \beta,
      \end{equation}
      where $V_i^{\text{DUAL}}(x_k,T_{\delta_k}\ii))$ is the value of the dual at the point $T_{\delta_k}\ii$.
\end{enumerate}

\begin{remark}
  \begin{enumerate}
    \item Both \eqref{StoppingCondition_Fstar} and \eqref{StoppingCondition_Dual} are termination criteria for the iterative method used for the inner loop to determine the inexact update $T_{\delta}\ii(x_k)$. They show that the error bound is indeed \emph{implementable}, and that the inexactness parameters $\alpha$ and $\beta$ relate to the stopping tolerance used in the inner loop of ICD.
        \item Selecting an appropriate stopping criterion and tolerance (i.e., deciding upon numerical values for $\alpha$ and $\beta$) is a problem that frequently arises when solving optimization problems, and the decision is left to the discretion of the user.
        \item It may be possible to find other stopping conditions such that \eqref{Td_def} is verifiable. Moreover, it may be possible for the algorithm to converge \emph{in practice} if a stopping condition is used for which \eqref{Td_def} cannot be checked.
  \end{enumerate}
\end{remark}

\subsection{Technical result}
\label{S_technicalresult}
The following result plays a key role in the complexity analysis of ICD.

\begin{theorem}\label{Theorem1}
     Fix $x_0 \in \R^N$ and let $\{x_k\}_{k \geq 0}$ be a sequence of random vectors in $\R^N$ with $x_{k+1}$ depending on $x_k$ only. Let $\varphi:\R^N \to \R$ be a nonnegative function, define $\xi_k \eqdef \varphi(x_k)$ and assume that $\{\xi_k\}_{k \geq 0}$ is nonincreasing. Further, let $\rho \in (0,1)$, $\epsilon>0$ and $\alpha,\beta \geq 0$ be such that one of the following two conditions holds:
\begin{enumerate}
     \item[(i)]  $\mathbb{E}[\xi_{k+1}\;|\;x_k] \leq (1+\alpha) \xi_k - \tfrac{\xi_k^2}{c_1} + \beta$, for all $k\geq 0$, where
     $c_1>0$,\\ $\tfrac{c_1}{2}\left(\alpha + \sqrt{\alpha^2 + \tfrac{4\beta}{c_1 \rho}}\right) <  \epsilon < \min\{(1+\alpha)c_1,\xi_0\} $ and $\sigma \eqdef \sqrt{\alpha^2 + \tfrac{4\beta}{c_1}} < 1$;
\item[(ii)]  $\mathbb{E}[\xi_{k+1}\;|\;x_k] \leq \left(1 + \alpha - \tfrac{1}{c_2}\right) \xi_k + \beta,$ for all $k\geq 0$ for which $\xi_k \geq \epsilon$, \\ where $\alpha c_2 < 1\leq (1+\alpha)c_2$, and $\tfrac{\beta c_2}{\rho(1 - \alpha c_2)} < \epsilon < \xi_0$.
\end{enumerate}
If $(i)$ holds and we define $u \eqdef \tfrac{c_1}{2}(\alpha + \sigma)$ and choose
\begin{equation}
\label{Thm_Ki}
     K\geq \frac{c_1}{\epsilon  - \alpha c_1} \log \left( \frac{\epsilon - \frac{\beta c_1}{\epsilon - \alpha c_1}}{\epsilon\rho - \frac{\beta c_1}{\epsilon - \alpha c_1}}\right) + \min\left\{   \frac{1}{\sigma} \log \left(\frac{\xi_0-u}{\epsilon-u} \right) , \frac{c_1}{\epsilon-u} -\frac{c_1}{\xi_0-u}   \right\} + 2,
\end{equation}
(where the second term in the minimum is chosen if $\sigma=0$),
or if $(ii)$ holds and we choose
\begin{equation}
\label{Thm_Kii}
K \geq \frac{c_2}{1-\alpha c_2} \log \left(\frac{\xi_0 - \frac{\beta c_2}{1-\alpha c_2}}{\epsilon \rho - \frac{\beta c_2}{1-\alpha c_2}}\right),
\end{equation}
then $\Prob(\xi_{K}\leq \epsilon) \geq 1- \rho$.
\end{theorem}
\begin{proof}
First notice that the thresholded sequence $\{\xi_k^{\epsilon}\}_{k \geq 0}$ defined by
\begin{equation}
\label{xikeps}
     \xi_k^{\epsilon} =
\begin{cases}
0, & \text{if } \; \xi_k < \epsilon,\\
\xi_k, & \text{otherwise,}
\end{cases}
\end{equation}
satisfies $\xi_k^{\epsilon} > \epsilon \Leftrightarrow \xi_k > \epsilon$. Therefore, by  Markov's inequality, $\Prob(\xi_k > \epsilon) =\Prob(\xi_k^{\epsilon} > \epsilon) \leq \frac{\E[\xi_k^{\epsilon}]}{\epsilon}$. Letting $\theta_k \eqdef \E[\xi_k^{\epsilon}]$, it thus suffices to show that
\begin{equation}
\label{theta_prob_result}
     \theta_K \leq \epsilon \rho.
\end{equation}
(The rationale behind this ``thresholding trick'' is that the sequence $\E[\xi_k^\epsilon]$ decreases faster than $\E[\xi_k]$ and hence will reach $\epsilon \rho$ sooner.) Assume now that (i) holds. It can be shown (for example, see Theorem 1 of \cite{Richtarik12} for the case $\alpha=\beta = 0$) that
\begin{equation}
\label{relations}
     \E[\xi_{k+1}^{\epsilon}\;|\;x_k] \leq (1+\alpha)\xi_{k}^{\epsilon} - \tfrac{(\xi_{k}^{\epsilon})^2}{c_1} + \beta,
\quad \E[\xi_{k+1}^{\epsilon}\;|\;x_k] \leq \left(1 + \alpha - \tfrac{\epsilon}{c_1}\right)\xi_{k}^{\epsilon} + \beta.
\end{equation}

By taking expectations in \eqref{relations} (in $x_k$) and using Jensen's inequality, we obtain
\begin{eqnarray}
     \label{relations_1} \theta_{k+1} &\leq& (1+\alpha)\theta_k - \tfrac{\theta_k^2}{c_1} + \beta, \quad k\geq 0,\\
\label{relations_2}\theta_{k+1} &\leq& \Big(1+\alpha-\tfrac{\epsilon}{c_1}\Big)\theta_k + \beta, \quad k\geq 0.
\end{eqnarray}

Notice that \eqref{relations_1} is better than \eqref{relations_2} precisely when $\theta_k > \epsilon$. It is easy to see that the inequality $(1+\alpha)\theta_k - \tfrac{\theta_k^2}{c_1} + \beta \leq \theta_k$ holds if and only if $\theta_k \geq u$. In other words, \eqref{relations_1} leads to $\theta_{k+1}$ that is better than $\theta_k$  only for  $\theta_k \geq u$. We will now compute $k=k_1$ for which $u <\theta_k \leq \epsilon$. Inequality \eqref{relations_1} can be equivalently written as
\begin{equation}\label{eq:shift}\theta_{k+1} - u \leq (1-\sigma)(\theta_k - u) - \frac{(\theta_k-u)^2}{c_1}, \quad k \geq 0.
\end{equation}
where $\sigma<1$. Writing \eqref{relations_1} in the form \eqref{eq:shift} eliminates the constant term $\beta$, which allows us to provide a simple analysis. (Moreover, this ``shifted'' form leads to a better result; see the remarks after the Theorem for details.) Letting $\hat{\theta}_k \eqdef \theta_k -u$, by monotonicity we have $\hat{\theta}_{k+1}\hat{\theta}_k \leq \hat{\theta}_k^2$, whence
\begin{equation}
\label{Thm1_eqn}
     \frac{1-\sigma}{\hat{\theta}_{k+1}} - \frac{1}{\hat{\theta}_k} = \frac{(1-\sigma)\hat{\theta}_k - \hat{\theta}_{k+1}}{\hat{\theta}_{k+1}\hat{\theta}_k} \geq \frac{(1-\sigma)\hat{\theta}_k - \hat{\theta}_{k+1}}{\hat{\theta}_k^2} \overset{\eqref{eq:shift}}{\geq} \frac{1}{c_1}.
\end{equation}
If we choose $r\in \{1,\tfrac{1}{1-\sigma}\}$, then
\[ \frac{1}{\hat{\theta}_k}
\overset{\eqref{Thm1_eqn}}{\geq} r  \left(\frac{1}{\hat{\theta}_{k-1}} + \frac{1}{c_1}\right) \geq  r^k \frac{1}{\hat{\theta}_0} + \frac{1}{c_1} \sum_{j=1}^{k} r^j
= \begin{cases}r^k \left(\frac{1}{\xi_0-u} + \frac{1}{c_1\sigma}\right) - \frac{1}{c_1\sigma}, & r=\tfrac{1}{1-\sigma},\\
\frac{1}{\xi_0-u} + \frac{k}{c_1}, & r=1.\end{cases}
\]
In particular, using the above estimate with  $r=1$ \emph{and} $r=\tfrac{1}{1-\sigma}$ gives
\begin{equation}\label{eq:k_1-eps}\hat{\theta}_{k_1} \leq \epsilon - u \qquad (\text{and hence } \theta_{k_1} \leq \epsilon)
\end{equation}
for
\begin{equation}\label{eq:k_1x}k_1 := \min\left\{\left \lceil \log \left(\frac{\frac{1}{\epsilon-u} + \frac{1}{c_1\sigma}}{\frac{1}{\xi_0-u} + \frac{1}{c_1\sigma}}\right)/ \log \left(\frac{1}{1-\sigma}\right) \right\rceil, \left\lceil \frac{c_1}{\epsilon-u}  -\frac{c_1}{\xi_0-u}  \right\rceil \right\},\end{equation}
where the left term in \eqref{eq:k_1x} applies when $\sigma>0$ only.

Applying the inequalities (i)~$\lceil t \rceil \leq 1 + t$; (ii)~$\log(\tfrac{1}{1-t})\geq t$ (holds for $0<t<1$; we use the inverse version, which is surprisingly tight for small $t$); and (iii)~the fact that $t \mapsto \tfrac{C+t}{D+t}$ is decreasing on $[0,\infty)$ if $C\geq D > 0$, we arrive at the following bound

\begin{equation}\label{eq:k1}k_1 \geq 1 + \min\left\{   \frac{1}{\sigma} \log \left(\frac{\xi_0-u}{\epsilon-u} \right) ,   \frac{c_1}{\epsilon-u}  -\frac{c_1}{\xi_0-u} \right\}.\end{equation}

Letting $\gamma\eqdef 1-\tfrac{\epsilon - \alpha c_1}{c_1}$ (notice that $\gamma\in (0,1)$), for any $k_2\geq 0$ we have
\begin{eqnarray}
\label{xxx}
      \theta_{k_1 +  k_2}
      &\overset{\eqref{relations_2}}{\leq}& \gamma \theta_{ k_1 +  k_2 - 1} + \beta \leq    \gamma^{k_2} \theta_{k_1 } + \beta (\gamma^{ k_2 -1} +\gamma^{ k_2 -2}+ \cdots + 1) \notag \\
      &\overset{\eqref{eq:k_1-eps}}{\leq}& \gamma^{k_2} \epsilon + \beta\frac{1-\gamma^{k_2}}{1-\gamma} = \gamma^{k_2}\left(\epsilon - \frac{\beta}{1-\gamma}\right) + \frac{\beta}{1-\gamma}.\label{eq:0987hhh}
\end{eqnarray}
In \eqref{xxx}, notice that the second to last term can be made as small as we like (by taking $k_2$ large), but we can never force $\theta_{k_1 +  k_2} \leq \tfrac{\beta}{1-\gamma}$. Therefore, in order to establish \eqref{theta_prob_result}, we need to ensure that $\tfrac{\beta c_1}{\epsilon - \alpha c_1} < \epsilon\rho$. Rearranging this gives the condition $\tfrac{c_1}{2}(\alpha + \sqrt{\alpha^2 +\frac{4\beta}{c_1 \rho}}) < \epsilon$, which holds by assumption. Now we can find $k_2$ for which the right hand side in \eqref{eq:0987hhh} is at most $\epsilon \rho$:
\begin{equation}\label{eq:k2A}k_2 \eqdef \left \lceil \log \left(\frac{\epsilon - \tfrac{\beta}{1-\gamma}}{\epsilon \rho - \tfrac{\beta}{1-\gamma}}\right)/ \log \left(\frac{1}{\gamma} \right) \right \rceil
\leq 1+ \frac{c_1}{\epsilon  - \alpha c_1} \log \left( \frac{\epsilon - \frac{\beta c_1}{\epsilon - \alpha c_1}}{\epsilon\rho - \frac{\beta c_1}{\epsilon - \alpha c_1}}\right).
\end{equation}

In view of \eqref{theta_prob_result}, it is enough to take $K = k_1+k_2$ iterations. The expression in \eqref{Thm_Ki} is obtained by adding the upper bounds on $k_1$ and $k_2$ in \eqref{eq:k1}  and \eqref{eq:k2A}.

Now assume that property ($ii$) holds. By a similar argument as that leading to \eqref{relations}, we obtain
\begin{eqnarray*}
\label{Thm1_ii_solve}
     \theta_K \leq \left(1-\tfrac{1-\alpha c_2}{c_2}\right)\theta_{K-1} + \beta
&\leq& \left(1-\tfrac{1-\alpha c_2}{c_2}\right)^K\theta_0 + \beta\sum_{j=0}^{K-1}\left(1-\tfrac{1-\alpha c_2}{c_2}\right)^j\\
&\leq & \left(1-\tfrac{1-\alpha c_2}{c_2}\right)^K\left(\theta_0 - \tfrac{\beta c_2}{1-\alpha c_2}\right) + \tfrac{\beta c_2}{1-\alpha c_2} \overset{\eqref{Thm_Kii}}{\leq} \epsilon \rho.
\end{eqnarray*}
The proof follows by taking $K$ given by \eqref{Thm_Kii}.
\end{proof}

Let us now comment on several aspects of the above result:
\begin{enumerate}
\item \emph{Usage.} We will use Theorem~\ref{Theorem1} to finish the proofs of the complexity results in Section~\ref{Section_CC}; with $\xi_k = \varphi(x_k) := F(x_k)-F^*$, where $\{x_k\}$ is the random process generated by ICD.

\item \emph{Monotonicity and Nonnegativity.}  Note that the monotonicity assumption in Theorem~\ref{Theorem1} is for the choice of $x_k$ and $\varphi$ described in 1) satisfied due to \eqref{eq:111}. Nonnegativity is satisfied automatically since $F(x_k) \geq F^*$ for all $x_k$.

\item \emph{Best of two.} In \eqref{eq:k_1x}, we notice that the first term applies when $\sigma >0$ only. If $\sigma = 0$, then $u=0$, and subsequently the second term in \eqref{eq:k_1x} applies, which corresponds to the exact case. Notice that if $\sigma>0$ is very small (so $u \neq 0$), the iteration complexity result still may be better if the second term is used.

\item \emph{Generalization.} Note that for $\alpha = \beta=0$, \eqref{Thm_Ki} recovers $\frac{c_1}{\epsilon}(1 + \log \frac{1}{\rho}) + 2 -  \tfrac{c_1}{\xi_0}$, which is the result proved in Theorem~1(i) in \cite{Richtarik12}, while \eqref{Thm_Kii} recovers $c_2 \log((F(x_0)-F^*)/\epsilon \rho)$, which is the result proved in Theorem~1(ii) in \cite{Richtarik12}. Since the last term in \eqref{Thm_Ki} is negative, the theorem holds also if we ignore it. This is what we have done, for simplicity, in Table~\ref{Table_Comparison}.

\item \emph{High accuracy with high probability.} In the exact case, the iteration complexity results hold for any error tolerance $\epsilon >0$ and confidence $\rho \in (0,1)$. However, in the inexact case, there are restrictions on the choice of $\rho$ and $\epsilon$ for which we can guarantee the result $\Prob(F(x_k)-F^*\leq \epsilon)\geq 1-\rho$. Table~\ref{Table_epsrho} gives conditions on $\alpha$ and $\beta$ under which arbitrary confidence level (i.e., small $\rho$) and accuracy (i.e., small $\epsilon$) is achievable. For instance, if Theorem~\ref{Theorem1}(ii) is used, then one can achieve arbitrary accuracy only if $\beta = 0$, but arbitrary confidence under no assumptions on $\alpha$ and $\beta$. The situation for part (i) is worse: $\epsilon$ is lower bounded by a positive expression that involves $\rho$, unless $\alpha=\beta=0$.
       \begin{table}[!h]\centering
         \begin{tabular}{|c | c|c |}
         \hline
         & Theorem \ref{Theorem1}(i) & Theorem \ref{Theorem1}(ii)\\
         \hline
         $\epsilon$ can be arbitrarily small if & $\alpha = \beta = 0$ & $\beta = 0$\\
         \hline
         $\rho$ can be arbitrarily small if & $\beta = 0$ & any $\alpha, \;\beta$\\
         \hline
        \end{tabular}
        \label{Table_epsrho}
        \caption{The conditions under which arbitrary confidence $\rho$ and accuracy $\epsilon$ are attainable.}
       \end{table}

%
%

\item \emph{Two lower bounds on $\epsilon$.} The inequality $\epsilon > \tfrac{c_1}{2}\left(\alpha + \sqrt{\alpha^2 + \tfrac{4\beta}{\rho c_1}}\right)$ (see part (i) of Theorem~\ref{Theorem1}) is equivalent to $\epsilon > \tfrac{\beta c_1}{\rho(\epsilon - \alpha c_1)}$. Note the similarity of the last expression and the lower bound on $\epsilon$ in part (ii) of the theorem. We can see that the lower bound on $\epsilon$ is smaller (and hence, is less restrictive) in (ii) than in (i), provided that $c_1=c_2$.

  \item \emph{Two analyses.} It can be seen that analyzing the ``shifted'' form \eqref{eq:shift} leads to a better result than analyzing \eqref{relations_1} directly, even when $\beta = 0$. Consider the case $\beta=0$, so that $\sigma=\alpha$ and $u=\alpha c_1$. From equation \eqref{Thm1_eqn}
$\theta_{k+1} \leq A \eqdef \alpha c_1 + (1-\alpha)/(\tfrac{1}{\theta_k-\alpha c_1} + \tfrac{1}{c_1}),$
whereas analyzing equation \eqref{relations_1} directly yields
$\theta_{k+1} \leq B \eqdef (1+\alpha)/(\tfrac{1}{\theta_k} + \tfrac{1}{c_1}).$ It can be shown that $A\leq B$, with equality if $\alpha = 0$.

\end{enumerate}


\section{Complexity Analysis: Convex Composite Objective}\label{Section_CC}
The following function plays a central role in our analysis:
\begin{equation}
\label{H}
     H(x,T) \eqdef f(x) + \langle \nabla f(x) ,T \rangle + \tfrac{1}{2} \|T\|_{\Lip}^2 + \Psi(x+T).
\end{equation}
Comparing \eqref{H} with \eqref{Vi} using \eqref{eq:decomposition}, \eqref{U_i}, \eqref{block_gradient}, \eqref{S2_separable_psi} and \eqref{S_Norms_1} we get
\begin{equation}
\label{H2}
     H(x,T) = f(x) + \sum_{i=1}^n V_i(x,T^{(i)}).
\end{equation}

It will be useful to establish inequalities relating $H$ evaluated at the vector of exact updates $T_0(x)$ and $H$ evaluated at the vector of inexact updates $T_\delta(x)$.
\begin{lemma}  For all $x\in \R^N$ and $\delta \in \R^n_+$,
\begin{equation}
\label{eq:222}
     H(x,T_0(x)) \leq H(x,T_{\delta}(x)) \leq  H(x,T_0(x)) + \sum_{i=1}^n \delta^{(i)}.
\end{equation}
\end{lemma}
{\em Proof:}
\begin{eqnarray*}
H(x,T_0(x))
&\overset{\eqref{H2}}{=}& f(x) + \sum_{i=1}^n  V_i(x,T_0^{(i)}(x)) \overset{\eqref{Td_def}}{=} f(x) + \sum_{i=1}^n  \min_{t \in \R^{N_i}} V_i(x,t)\\
&\leq& f(x) + \sum_{i=1}^n V_i(x,T^{(i)}_{\delta}(x)) \overset{\eqref{H2}}{=}  H(x,T_{\delta}(x))\\
&\overset{\eqref{Td_def}}{\leq} & f(x) + \sum_{i=1}^n \Big(\delta^{(i)} + \min_{t \in \R^{N_i}} V_i(x,t)\Big)   \overset{\eqref{H2}}{=}  H(x,T_0(x)) + \sum_{i=1}^n \delta^{(i)}. \quad \qed
\end{eqnarray*}

The following Lemma provides an upper bound on the expected distance between the  current and optimal objective value in terms of the function $H$.
\begin{lemma}
\label{Lemma_HpF} For $x,T \in \R^N$, let $x_+(x,T)$ be the random vector equal to $x + U_i T^{(i)}$ with probability $\tfrac{1}{n}$ for each $i \in \{1,2,\dots,n\}$.
Then
\begin{eqnarray*}
     \mathbb{E}[F(x_+(x,T)) - F^* \;|\; x] \leq \tfrac{1}{n}(H(x,T) - F^*) +  \tfrac{n-1}{n}(F(x) - F^*).
\end{eqnarray*}
\end{lemma}
{\em Proof:}
\begin{eqnarray*}
     \mathbb{E}[F(x_+(x,T)) \;|\; x] &=& \sum_{i=1}^n \tfrac{1}{n} F(x + U_iT^{(i)})\\
     &\overset{\eqref{S2_upperbound_F}}{\leq}& \tfrac{1}{n} \sum_{i=1}^n [f(x) + V_i(x,T^{(i)}) + \psi_i(x)]\\
     &\overset{\eqref{H2}+\eqref{psi}}{=}& \tfrac{1}{n}H(x,T) + \tfrac{n-1}{n}f(x) + \tfrac{1}{n} \sum_{i=1}^n \sum_{j \neq i} \Psi_j(x^{(j)})\\
     &=& \tfrac{1}{n}H(x,T) + \tfrac{n-1}{n}F(x). \qquad \qed
\end{eqnarray*}

Note that if $x = x_k$ and $T = T_{\delta}(x_k)$, then $x_+(x,T) = x_{k+1}$, as produced by Algorithm~\ref{ICD}. The following Lemma, which provides an upper bound on $H$, will be used repeatedly throughout the remainder of this paper.
\begin{lemma}
\label{Lemma_3}
For all $x \in \dom F$ and $\delta \in \R^n_+$ (letting $\Delta = \sum_i \delta^{(i)}$), we have
\begin{equation}
\label{Lemma_3eqn}
     H(x,T_{\delta}(x)) \leq \Delta + \min_{y \in \R^N} \left\{F(y) + \tfrac{1-\mu_f(\Lip)}{2}\|y-x\|_{\Lip}^2\right\}.
\end{equation}
\end{lemma}
{\em Proof:}
     \begin{eqnarray}
\notag
H(x,\TT(x))  &\overset{\eqref{eq:222}}{\leq}& \Delta + \min_{T \in \mathbb{R}^N} H(x,T)  \\
\notag
&=& \Delta + \min_{y\in \mathbb{R}^N } H(x,y-x) \qquad \qquad (\text{where } y = x+T)\\
\notag
&\overset{\eqref{H}}{=}& \Delta +\min_{y\in \mathbb{R}^N} \{f(x) + \langle \nabla f(x),y-x\rangle + \Psi(y) + \tfrac{1}{2} \|y-x\|_{\Lip}^2 \}\\
\notag
&\overset{\eqref{strongly_convex_1}}{\leq}& \Delta + \min_{y\in \mathbb{R}^N}\{ f(y) -  \tfrac{\mu_f(\Lip)}{2} \|y-x\|_{\Lip}^2 + \Psi(y) + \tfrac{1}{2} \|y-x\|_{\Lip}^2\}. \qquad \qed
\end{eqnarray}

\subsection{Convex case}

Now we need to estimate $H(x,\TT(x))-F^*$ from above in terms of $F(x)-F^*$.

\begin{lemma}
\label{Lemma_HF}
     Fix $x^* \in X^*$, $x \in \dom F$, $\delta \in \R^n_+$ and let $R = \|x-x^*\|_{\Lip}$  and $\Delta=\sum_i \delta^{(i)}$. Then
\begin{equation}
\label{result_L2}
     H(x,\TT(x)) - F^* \leq \Delta +
\begin{cases}
(1-\frac{F(x) - F^*}{2R^2})(F(x) - F^*), & \text{if } F(x) - F^* \leq R^2,\\
 \frac{1}{2}R^2 < \frac{1}{2}(F(x) - F^*), & \text{otherwise}.
\end{cases}
\end{equation}
\end{lemma}
{\em Proof:}
Because strong convexity is not assumed, $\mu_f(\Lip) = 0$, so
\begin{eqnarray*}
H(x,\TT(x))  &\overset{\eqref{Lemma_3eqn}}{\leq}& \Delta + \min_{y\in \mathbb{R}^N}\{ F(y) +  \tfrac{1}{2} \|y-x\|_{\Lip}^2\}\\
&\leq& \Delta +\min_{\lambda \in [0,1]} \{F(\lambda x^* + (1-\lambda)x) +  \tfrac{\lambda^2}{2} \|x-x^*\|_{\Lip}^2\}\\
&\leq& \Delta + \min_{\lambda \in [0,1]} \{F(x) - \lambda(F(x)-F^*)+  \tfrac{\lambda^2}{2} R^2\}.
\end{eqnarray*}
Minimizing in $\lambda$ gives $\lambda^* = \min \{1, (F(x)-F^*) /R^2 \}$ and the result follows. \qquad \qed

We now state the main complexity result of this section, which bounds the number of iterations sufficient for ICD used with uniform probabilities to decrease the value of the objective to within $\epsilon$ of the optimal value with probability at least $1-\rho$.

\begin{theorem}\label{Thm:C-N} Choose an initial point $x_0\in \R^N$ and let $\{x_k\}_{k\geq 0}$ be the random iterates generated by ICD applied to problem \eqref{F}, using uniform probabilities $p_i=\frac{1}{n}$ and inexactness parameters $\delta_k^{(1)},\dots, \delta_k^{(n)}\geq 0$ that satisfy \eqref{eq:alpha-beta} for $\alpha,\beta \geq 0$. Choose target confidence $\rho \in (0,1)$ and error tolerance $\epsilon>0$ so that one of the following two conditions hold:
\begin{itemize}
     \item[(i)] $\tfrac{c_1}{2}(\alpha + \sqrt{\alpha^2 + \tfrac{4\beta}{c_1 \rho}}) < \epsilon < F(x_0) - F^*$ and $\alpha^2 + \tfrac{4\beta}{c_1 }<1$, where $c_1 = 2n\max\{\mathcal{R}_{\Lip}^2(x_0),F(x_0) - F^*\}$,
\item[(ii)] $\tfrac{\beta c_2}{\rho(1-\alpha c_2)} < \epsilon <  \min\{\mathcal{R}_{\Lip}^2(x_0),F(x_0) - F^*\}$, where $c_2 = \tfrac{2n\mathcal{R}_{\Lip}^2(x_0)}{\epsilon}$ and $\alpha c_2 < 1$.
\end{itemize}
If (i) holds and we choose $K$ as in \eqref{Thm_Ki}, or if (ii) holds and we choose $K$ as in \eqref{Thm_Kii}, then $\Prob(F(x_K)-F^* \leq \epsilon) \geq 1-\rho$.
\end{theorem}

\begin{proof}
Since $F(x_k)\leq F(x_0)$ for all $k$ by \eqref{eq:111}, we have $\|x_k - x^*\|_{\Lip} \leq \mathcal{R}_{\Lip}(x_0)$ for all $k$ and $x^* \in X^*$. Using Lemma~\ref{Lemma_HpF} and Lemma~\ref{Lemma_HF}, and letting $\xi_k\eqdef F(x_k) - F^*$, we have
\begin{eqnarray}
\mathbb{E}[\xi_{k+1} \;|\; x_k] &\leq& \bar{\delta}_k + \tfrac{1}{n} \max\left\{1 - \tfrac{\xi_k}{2\|x_k - x^*\|_{\Lip}^2},\tfrac{1}{2} \right\}\xi_k + \tfrac{n-1}{n}\xi_k \\
\notag
&=& \bar{\delta}_k + \max\left\{1 - \tfrac{\xi_k}{2n\|x_k - x^*\|_{\Lip}^2},1-\tfrac{1}{2n} \right\}\xi_k\\
\label{Thm2_last}
&\leq& \bar{\delta}_k + \max\left\{1 - \tfrac{\xi_k}{2n\mathcal{R}_{\Lip}^2(x_0)},1-\tfrac{1}{2n} \right\}\xi_k.
\end{eqnarray}
Consider case (i). From \eqref{Thm2_last} and \eqref{eq:alpha-beta} we obtain
\begin{equation}
\mathbb{E}[\xi_{k+1}\;|\;x_k] \leq  \bar{\delta}_k + \Big(1 - \tfrac{\xi_k}{c_1}\Big)\xi_k \leq  (1+\alpha) \xi_k - \tfrac{\xi_k^2}{c_1} + \beta,
\end{equation}
and the result follows by applying Theorem~\ref{Theorem1}(i). Now consider case (ii). Notice that if $\xi_k \geq \epsilon$, then \eqref{Thm2_last} together with \eqref{eq:alpha-beta}, imply that
\begin{equation*}
     \mathbb{E}[\xi_{k+1}\;|\;x_k] \leq  \bar{\delta}_k + \max\left\{1 - \tfrac{\epsilon}{2n\mathcal{R}_{\Lip}^2(x_0)},1-\tfrac{1}{2n} \right\}\xi_k \leq  \left(1 + \alpha- \tfrac{1}{c_2}\right)\xi_k + \beta.
\end{equation*}
The result follows by applying Theorem~\ref{Theorem1}(ii).
\end{proof}

\subsection{Strongly convex case}

Let us start with an auxiliary result.

\begin{lemma}
\label{Lemma_stronglyconvexHF}
     Let $F$ be strongly convex with respect to $\| \cdot \|_{\Lip}$ with $\mu_f(\Lip) + \mu_{\Psi}(\Lip) >0$. Then for all $x\in \dom F$ and $\delta\in \R^n_+$, with $\Delta= \sum_i \delta^{(i)}$, we have
\begin{eqnarray*}
     H(x,\TT(x)) - F^* \leq \Delta + \left(\tfrac{1-\mu_f(\Lip)}{1+\mu_\Psi(\Lip)}\right)(F(x)-F^*).
\end{eqnarray*}
\end{lemma}
{\em Proof:}
Let $\mu_f = \mu_f(\Lip)$, $\mu_\Psi = \mu_\Psi(\Lip)$ and $\lambda^* = (\mu_f + \mu_\Psi)/(1+\mu_\Psi) \leq 1$. Then,
\begin{eqnarray*}
   H(x, \TT(x))&\overset{\eqref{Lemma_3eqn}}{\leq} & \Delta +  \min_{y \in \R^N} \{ F(y) +  \tfrac{1-\mu_f}{2} \|y-x\|_{\Lip}^2\}\\
& \leq& \Delta +\min_{\lambda \in [0,1]} \{F(\lambda x^* + (1-\lambda)x) +  \tfrac{(1-\mu_f) \lambda^2}{2} \|x-x^*\|_{\Lip}^2\}\\
&\overset{\eqref{strongly_convex_4}+\eqref{strongly_convex_3}}{\leq}& \Delta + \min_{\lambda \in [0,1]}\{ \lambda F^* +(1- \lambda)F(x) +  \tfrac{(1-\mu_f)\lambda^2-(\mu_f+\mu_\Psi)\lambda(1-\lambda)}{2} \|x-x^*\|_{\Lip}^2\}\\
&\leq& \Delta + F(x) - \lambda^* (F(x)-F^*).
\end{eqnarray*}
The last inequality follows from the fact that $(\mu_f +\mu_\Psi)(1-\lambda^*) - (1-\mu_f)\lambda^* = 0$. It remains to subtract $F^*$ from both sides of the final inequality. \quad \qed

We can now estimate the number of iterations needed to decrease a strongly convex objective $F$ within $\epsilon$ of the optimal value with high probability.
\begin{theorem}
\label{Thm_ICD_strong}
     Let $F$ be strongly convex with respect to the norm $\| \cdot \|_{\Lip}$ with $\mu_f(\Lip) + \mu_\Psi(\Lip)>0$ and let $\mu \eqdef \tfrac{\mu_f(\Lip)+\mu_\Psi(\Lip)}{1+\mu_\Psi(\Lip)}$. Choose an initial point $x_0\in \R^N$ and let $\{x_k\}_{k\geq 0}$, be the random iterates generated by ICD applied to problem \eqref{F}, used with uniform probabilities $p_i=\tfrac{1}{n}$ for $i=1,2,\dots,n$ and inexactness parameters $\delta_k^{(1)},\dots,\delta_k^{(n)}\geq 0$ satisfying \eqref{eq:alpha-beta}, for $0\leq \alpha< \frac{\mu}{n}$ and $\beta\geq 0$. Choose confidence level  $\rho \in (0,1)$ and  error tolerance $\epsilon$ satisfying $\tfrac{\beta n}{\rho(\mu-\alpha n)} <\epsilon < F(x_0)-F^*$. Then for $K$ given by \eqref{Thm_Kii}, we have $\Prob(F(x_K) - F^* \leq \epsilon) \geq 1-\rho$.
\end{theorem}
\begin{proof}
Letting $\xi_{k} = F(x_k)-F^*$, we have
\begin{eqnarray*}\E[\xi_{k+1}\;|\; x_k] &\overset{(\text{Lemma~\ref{Lemma_HpF})}}{\leq} & \tfrac{1}{n} (H(x_k,T_{\delta_k}(x_k)) - F^*)  + \tfrac{n-1}{n}\xi_k\\
& \overset{(\text{Lemma~\ref{Lemma_stronglyconvexHF}})}{\leq} &
 \bar{\delta}_k + \tfrac{1}{n}\left(\tfrac{1-\mu_f(\Lip)}{1+\mu_\Psi(\Lip)} \xi_k\right)   + \tfrac{n-1}{n}\xi_k \\
&\overset{\eqref{eq:alpha-beta}}{\leq}&  \left(1 + \alpha -\tfrac{\mu}{n}\right)\xi_k + \beta.
\end{eqnarray*}
By \eqref{eq_lipp1}, $\mu \leq 1$, and the result follows from Theorem \ref{Theorem1}(ii) with $c_2 = \frac{n}{\mu}$.
\end{proof}

\section{Complexity Analysis: Smooth Objective}
\label{Section_CS}
In this section we provide simplified iteration complexity results when the objective function is smooth ($\Psi \equiv 0$ so $F \equiv f$). Furthermore, we provide complexity results for arbitrary (rather than uniform) probabilities $p_i>0$.

\subsection{Convex case}
In the smooth exact case, we can write down a closed-form expression for the update:
\begin{eqnarray*}
     T_0^{(i)}(x) \overset{\eqref{Td_def}}{=} \arg \min_{t \in \R^{N_i}} V_i(x,t) \overset{\eqref{Vi}}{=} \arg \min_{t \in \R^{N_i}} \{\langle \nabla_i f(x), t \rangle + \tfrac{\Lip_i}{2}\|t \|_{(i)}^2\} = -\tfrac{1}{\Lip_i}B_i^{-1} \nabla_i f(x).
\end{eqnarray*}
Substituting this into $V_i(x,\cdot)$ yields
\begin{equation}
\label{Smooth_valVi}
     V_i(x,T_0^{(i)}(x)) =\langle \nabla_i f(x),T_0^{(i)}(x) \rangle + \tfrac{\Lip_i}{2}\|T_0^{(i)}(x) \|_{(i)}^2 = -\tfrac{1}{2\Lip_i}(\|\nabla_i f(x)\|_{(i)}^*)^2.
\end{equation}
We can now estimate the decrease in $f$ during one iteration of ICD:
\begin{eqnarray}
\notag
     f(x + U_i T_{\delta}^{(i)}(x)) - f(x) &\overset{\eqref{S2_upperbound}}{\leq}& \langle \nabla_i f(x), T_{\delta}^{(i)}(x) \rangle + \tfrac{\Lip_i}{2}\| T_{\delta}^{(i)}(x)\|_{(i)}^2\\
     & \overset{\eqref{Vi}}{=} & V_i(x,T_\delta^{(i)}(x)) \notag\\
     & \overset{\eqref{Td_def}}{\leq} & \min\{0,\delta^{(i)} + V_i(x,T_0^{(i)}(x))\} \notag\\
\label{S_SmoothConvex_fdiff}
&\overset{\eqref{Smooth_valVi}}{=}& \min\{0,\delta^{(i)} - \tfrac{1}{2\Lip_i}\big(\|\nabla_i f(x)\|_{(i)}^*\big)^2\}.
\end{eqnarray}
The main iteration complexity result of this section can be now established.
\begin{theorem}\label{Thm:C-S}
     Choose an initial point $x_0\in \R^N$ and let $\{x_k\}_{k\geq 0}$ be the random iterates generated by ICD applied to the problem of minimizing $f$, used with probabilities $p_1,\dots,p_n >0$ and inexactness parameters $\delta_k^{(1)},\dots, \delta_k^{(n)}\geq 0$ satisfying \eqref{eq:alpha-beta} for $\alpha,\beta \geq 0$, where $\alpha^2 +\tfrac{4\beta}{c_1}<1$ and $c_1 =2\mathcal{R}_{\Lip p^{-1}}^2(x_0)$. Choose target confidence $\rho \in (0,1)$, error tolerance $\epsilon$ satisfying $\tfrac{c_1}{2}(\alpha + \sqrt{\alpha^2 + \tfrac{4\beta}{c_1 \rho}}) < \epsilon < f(x_0)-f^*$,  and let the iteration counter $K$ be given by \eqref{Thm_Ki}. Then $\Prob(f(x_K) - f^* \leq \epsilon) \geq 1-\rho$.
\end{theorem}
\begin{proof}
     We first estimate the expected decrease of the objective function during one iteration of the method:
\begin{eqnarray}
\notag
     \E[f(x_{k+1})\;|\;x_k] &=& f(x_k) +\sum_{i=1}^n p_i [f(x_k+U_i T_{\delta_k}^{(i)}(x_k)) - f(x_k)]\\
\notag
&\overset{\eqref{S_SmoothConvex_fdiff}}{\leq}&  f(x_k) + \sum_{i=1}^n p_i \left( \delta_k^{(i)} - \tfrac{1}{2\Lip_i}\big(\|\nabla_i f(x_k)\|_{(i)}^*\big)^2\right)\\
\notag
&\overset{\eqref{S_Norms_1}}{=}& f(x_k) - \tfrac{1}{2}\big(\|\nabla f(x_k)\|_{lp^{-1}}^*\big)^2 + \sum_{i=1}^n p_i \delta_k^{(i)}\\
\label{eq_sc}
&\leq& f(x_k) - \tfrac{1}{2} \big(\|\nabla f(x_k)\|_{lp^{-1}}^*\big)^2 + \alpha(f(x_k)-f^*) + \beta.
\end{eqnarray}
Since $f(x_k) \leq f(x_0)$ for all $k$,
\begin{equation}
\label{eq_fR}
 f(x_k)-f^*\leq \max_{x^* \in X^*} \langle \nabla f(x_k) , x_k - x^*\rangle\leq \|\nabla f(x_k)\|_{lp^{-1}}^* \mathcal{R}_{lp^{-1}}(x_0).
\end{equation}
Substituting \eqref{eq_fR} into \eqref{eq_sc} we obtain
\begin{equation}
\label{To_rearrange}
     \E[f(x_{k+1})-f^*\;|\;x_k] \leq  f(x_k) - f^* - \tfrac{1}{2} \left( \tfrac{f(x_k) - f^*}{\mathcal{R}_{\Lip p^{-1}}(x_0)}\right)^2 + \alpha(f(x_k)-f^*) + \beta.
\end{equation}
It remains to apply Theorem~\ref{Theorem1}(i).
\end{proof}

\subsection{Strongly convex case}

In this section we assume that $f$ is strongly convex with respect to $\|\cdot\|_{\Lip \probs^{-1}}$ with convexity parameter $\mu_f(\Lip\probs^{-1})$. Using \eqref{strongly_convex_1} with $x = x_k$ and $y = x^*$, and letting $h = x^*-x_k$, we obtain
\begin{eqnarray}
    \notag
    f^* - f(x_k) &\geq& \langle \nabla f(x_k),h\rangle + \tfrac{\mu_f(\Lip \probs^{-1})}{2}\|h\|_{\Lip \probs^{-1}}^2\\
    \label{S_SmoothSConvex_1}
    &=& \mu_f(\Lip \probs^{-1})\left(\langle \tfrac{1}{\mu_f(\Lip \probs^{-1})}\nabla f(x_k),h\rangle + \tfrac{1}{2}\|h\|_{\Lip \probs^{-1}}^2\right).
\end{eqnarray}
By minimizing the right hand side of \eqref{S_SmoothSConvex_1}, and rearranging, we obtain
\begin{equation}
\label{S_SmoothSConvex_2}
     f(x_k)-f^* \leq \frac{1}{2\mu_f(\Lip \probs^{-1})}(\|\nabla f(x_k)\|_{\Lip \probs^{-1}}^*)^2.
\end{equation}
We can now give an efficiency estimate for the case of a strongly convex objective.

\begin{theorem}\label{Thm:SC-S} Let $f$ be strongly convex with respect to the norm $\|\cdot \|_{\Lip \probs^{-1}}$ with convexity parameter $\mu_f(\Lip \probs^{-1})>0$. Choose an initial point $x_0\in \R^N$ and let $\{x_k\}_{k\geq 0}$ be the random iterates generated by ICD applied to the problem of minimizing $f$, used with  probabilities $p_1,\dots,p_n>0$ and inexactness parameters $\delta_k^{(1)},\dots, \delta_k^{(n)}\geq 0$ that satisfy \eqref{eq:alpha-beta} for $0\leq\alpha< \mu_f(\Lip\probs^{-1})$ and $\beta \geq 0$. Choose the target confidence $\rho \in (0,1)$, let the target accuracy $\epsilon$ satisfy $\frac{\beta}{\rho(\mu_f(\Lip\probs^{-1}) - \alpha)} < \epsilon < f(x_0)-f^*$, let $c_2 = 1/\mu_f(\Lip \probs^{-1})$ and let iteration counter $K$ be as in \eqref{Thm_Kii}. Then $\Prob(f(x_K) - f^* \leq \epsilon) \geq 1- \rho$.
\end{theorem}
\begin{proof}
     The expected decrease of the objective function during one iteration of the method can be estimated as follows:
\begin{eqnarray*}
     \E[f(x_{k+1}) - f^*|x_k] &\overset{\eqref{eq_sc}}{\leq}& (1+\alpha)(f(x_k) -f^*) - \frac{1}{2}(\|\nabla f(x_k)\|_{\Lip \probs^{-1}}^*)^2 + \beta\\
&\overset{\eqref{S_SmoothSConvex_2}}{\leq}& (1+\alpha - \mu_f(\Lip \probs^{-1}))(f(x_k) -f^*)  + \beta
\end{eqnarray*}
It remains to apply Theorem~\ref{Theorem1}(ii) with $\varphi(x_k) = f(x_k) - f^*$ (and notice that $c_2 >1$ by \eqref{eq_lipp1}).
\end{proof}

\section{Practical aspects of an inexact update}
\label{Section_Practical}
The goal of the second part of this paper is to demonstrate the practical importance of employing an inexact update in the (block) ICD method.

\subsection{Solving smooth problems via ICD}

In the first part of this section we assume that $\Psi=0$, so the function $F(x) = f(x)$ is smooth and convex. In this case the overapproximation is
\begin{equation}
\label{overapprox_Bi}
      F(x_k+U_i t) = f(x_k+U_i t) \overset{\eqref{S2_upperbound}+ \eqref{S2_Norm_def}}{\leq} f(x_k) + \langle \nabla_i f(x_k), t\rangle + \tfrac{l_i}{2}\langle B_i t , t\rangle \equiv f(x_k) + V_i(x_k,t).
\end{equation}
Differentiating \eqref{overapprox_Bi} with respect to $t$ and setting the result to $0$, shows that determining the update to block $i$ at iteration $k$ is equivalent to solving the system of equations
\begin{equation}
\label{eqn_solving_Bi}
     B_i t = - \tfrac{1}{l_i}  \nabla_i f(x_k).
\end{equation}
Recall that $B_i$ is positive definite so the exact update is
\begin{equation}
\label{Equation_T0}
  T_0\ii(x_k) = - \tfrac{1}{l_i} B_i^{-1} \nabla_i f(x_k),
\end{equation}
and, as mentioned in Section \ref{S_VerifyVi}, $V_i(x_k,T_0\ii(x_k)) = 0$. Clearly, solving systems of equations is central to the block coordinate descent method in the smooth case.

Exact CD \cite{Richtarik11a} requires the exact update \eqref{Equation_T0}, which depends on the inverse of an $N_i \times N_i$ matrix.
A standard approach to solving for $T_0\ii(x_k)$ in \eqref{Equation_T0} is to form the Cholesky factors of $B_i$ followed by two triangular solves. This can be extremely expensive for medium $N_i$, or dense $B_i$.

The results in this work allow \eqref{eqn_solving_Bi} to be solved using an iterative method to find an \emph{inexact} update $T_{\delta_k}\ii(x_k)$. If we compute $t$ for which
$V_i(x_k,t) - V_i(x_k,T_0\ii(x_k)) = V_i(x_k,t) = \|B_i t - \tfrac{1}{l_i}\nabla_i f(x_k)\|_2^2 \leq \beta,$ then we terminate the iterative method and accept the inexact update $T_{\delta_k}\ii \equiv t$.\footnote{Note that if $f$ is a quadratic function corresponding to a consistent system of equations, then we could have used the more general stopping condition $\alpha F(x_k)+\beta$, because then we also know that $F^*=0$.}

Because $B_i$ is positive definite, a natural choice is to solve \eqref{overapprox_Bi} using conjugate gradients \cite{Hestenes52}. (This is the method we adopt in the numerical experiments presented in Section \ref{Section_Numerical}). It is widely accepted that using an iterative technique has many advantages over a direct method for solving systems of equations, so we expect that an inexact update can be determined quickly, and subsequently the overall ICD algorithm running time reduces. Moreover, applying a preconditioner to \eqref{overapprox_Bi} can enable even faster convergence of conjugate gradients. Finding good preconditioners is an active area of research; see for example \cite{Benzi05,Golub99,Gratton11}.

\subsubsection{A special case: a quadratic function}
\label{Section_quadratic}
A special case of the above is when we have the unconstrained quadratic minimization problem
\begin{equation}
\label{eq_fprob}
     \min_{x \in \R^N} f(x) = \tfrac{1}{2} \|A x - b\|_2^2,
\end{equation}
where $A \in \mathbb{R}^{M \times N}$, and $b \in \R^M$. In this case, the overapproximation \eqref{S2_upperbound} becomes
\begin{eqnarray}
\label{ub_exact}
  f(x + U_it) = \tfrac{1}{2}\|A(x+U_it) - b\|_2^2
  = f(x) + \langle \nabla_i f(x) ,t \rangle + \tfrac{1}{2}\langle A_i^TA_it ,t \rangle,
\end{eqnarray}
where $A_i = U_i A$.
Comparing \eqref{ub_exact} with \eqref{overapprox_Bi}, we see that in the quadratic case, \eqref{ub_exact} is an exact upper bound on $f(x + U_it)$ if we choose $l_i = 1$ and $B_i = A_i^TA_i$ for all blocks $i = 1,\dots,n$. The matrix $B_i$ is required to be (strictly) positive definite so $A_i$ is assumed to have full (column) rank.\footnote{If a block $\Ai$ does not have full column rank then we simply adjust our choice of $l_i$ and $B_i$ accordingly, although this means that we have an overapproximation to $f(x+U_it)$, rather than equality as in \eqref{ub_exact}.} Substituting $l_i = 1$ and $B_i = A_i^TA_i$ into \eqref{eqn_solving_Bi} gives
\begin{eqnarray}
\label{normal_eqi}
     A_i^TA_i t = - A_i^T(Ax-b).
\end{eqnarray}
Therefore, when ICD is applied to a problem of the form \eqref{eq_fprob}, the update is found by solving \eqref{normal_eqi}.

\subsection{Solving nonsmooth problems via ICD}

The nonsmooth case is not as simple as the smooth case, because the update subproblem will have a different form for each nonsmooth term $\Psi$. However, we will see that in many cases, the subproblem will have the same, or similar, form to the original objective function. We demonstrate this through the use of the following concrete examples.

\subsubsection{Group Lasso}

A widely studied optimization problem arising in statistics and machine learning is the so-called group lasso problem, which has the form
\begin{equation}\label{GroupLasso}
  \min_{x \in \R^N} \tfrac12\|Ax-b\|_2^2 + \lambda \sum_{i=1}^n \sqrt{d_i} \|x\ii\|_2,
\end{equation}
where $\lambda >0$ is a regularization parameter and $d_i$ for all $i$ is a weighting parameter that depends on the size of the $i$th block. Formulation \eqref{GroupLasso} fits the structure \eqref{F} with $f(x) = \tfrac12 \|Ax-b\|_2^2$ and $\Psi(x) = \sum_{i=1}^n \lambda \sqrt{d_i} \|x\ii\|_2$. It can be shown that choosing $B_i = A_i^TA_i$\footnote{Here we assume that $A_i^TA_i\succ 0$} and $l_i = 1$ for all $i$, satisfies the overapproximation \eqref{S2_upperbound_F} giving
 $ F(x_k+U_i t) \leq f(x_k) + \langle A_i^T r_k , t\rangle + \tfrac12\langle A_i^TA_it,t\rangle + \lambda \sqrt{d_i} \|x_k\ii + t\|_2$,
where $r_k = Ax_k-b$, so
\begin{equation}\label{GroupLassoVi}
  V_i(x_k,t) = \tfrac12\|A_it-r_k\|_2^2 + \lambda \sqrt{d_i} \|x_k\ii+t\|_2.
\end{equation}
We see that (after a simple change of variables) \eqref{GroupLassoVi} has the same form as the original problem \eqref{GroupLasso}. We can apply \emph{any} algorithm to approximately minimize \eqref{GroupLassoVi} that uses one of the stopping conditions described in Section \ref{S_VerifyVi}.

%
%

\section{Numerical Experiments}
\label{Section_Numerical}
In this section we present preliminary numerical results to demonstrate the practical performance of Inexact Coordinate Descent and compare the results with Exact Coordinate Descent. We note that a thorough practical investigation of Exact CD is given in \cite{Richtarik12} where its usefulness on huge-scale problems is evidenced. We do not intend to reproduce such results for ICD, rather, we investigate the affect of inexact updates compared with exact updates, which should be apparent on medium scale problems. We do this the full knowledge that if exact CD scales well to very large sizes (shown in \cite{Richtarik12}) then so too will ICD.

Each experiment presented in this section was implemented in {\sc{Matlab}} and run (under linux) on a desktop computer with a quad core i5-3470CPU, 3.20GHz processor with 24Gb of RAM.

\subsection{Problem description for a smooth objective}
\label{Section_Numerical_problemdescription}
In this numerical experiment, we assume that the function $F =f$ is quadratic \eqref{eq_fprob} and $\Psi = 0$. Further, as ICD can work with blocks of data, we impose block structure on the system matrix. In particular, we assume that the matrix $A$ has block angular structure. Matrices with this structure frequently arise in optimization, from optimal control, scheduling and planning problems to stochastic optimization problems, and exploiting this structure is an active area of research \cite{Castro11,Gondzio03,Schultz91}. To this end, we define

\begin{eqnarray}
\label{A_matrix}
     A = \left[\begin{array}{c}
  C \\
\hline
    D
 \end{array}\right]  \in \R^{M \times N},
\end{eqnarray}
with the partitioning
\begin{eqnarray}
\label{C}
     C = \begin{bmatrix}
  C_1  & & \\
& \ddots & \\
& &C_n\\
 \end{bmatrix}  \in \R^{m \times N},\quad
\label{D}
     D =
 \begin{bmatrix}
D_1  & \dots & D_n
 \end{bmatrix} \in \R^{\ell \times N} \quad
\text{and} \quad
\label{Ai_matrix}
\Ai  =  \begin{bmatrix}
\\ \Ci\\ \\ \Di
\end{bmatrix}\in \R^{M \times N_i}.
\end{eqnarray}
Moreover, we assume that each block $C_i \in \R^{M_i \times N_i}$, and the linking blocks $\Di \in \R^{\ell \times N_i}$. We assume that $\ell \ll N$, and that there are $n$ blocks with $m = \sum_{i=1}^n M_i$ so $M = m + \ell$, and $N = \sum_{i=1}^n N_i$.

Notice that if $D = \mathbf{0}$, where $\mathbf{0}$ is the $\ell \times N$ matrix of all zeros, then problem \eqref{eq_fprob} is completely (block) separable so it can be solved easily. The linking constraints $D$ make problem \eqref{eq_fprob} nonseparable, which makes it non-trivial to solve.

The system of equations \eqref{normal_eqi} must be solved at each iteration of ICD (where $B_i = \Ai^T\Ai = \Ci^T\Ci + \Di^T\Di$) because it determines the update to apply to the $i$th block. We solve this system \emph{inexactly} using an \emph{iterative method}. In particular we use the conjugate gradient method (CG) in the numerical experiments presented in this section.

It is well known that the performance of CG is improved by the use of an appropriate preconditioner. To this end, we compare ICD using CG with ICD using preconditioned CG (PCG).
If $M_i \geq N_i$ and rank($C_i) = N_i$, then the block $C_i^TC_i$ is positive definite  so we propose the preconditioner (for the $i$th system)
\begin{equation}
\label{preconditioner}
     \p  \eqdef \Ci^T\Ci.
\end{equation}

If $M_i< N_i$ then $\p$ is rank deficient and is therefore singular. In such a case, we perturb \eqref{preconditioner} by adding a multiple of the identity matrix, and propose the nonsingular preconditioner
\begin{eqnarray}
\label{preconditioner_perturbed}
     \hat{\p} =  \p + \rho I = \Ci^T\Ci + \rho  I,
\end{eqnarray}
where $\rho > 0$.

Applying the preconditioners (defined in \eqref{preconditioner} for $M_i\geq N_i$, and \eqref{preconditioner_perturbed} for $M_i<N_i$) to \eqref{normal_eqi}, should result in the system having better spectral properties than the original, and this will lead to faster convergence of the conjugate gradient algorithm. A full theoretical justification (eigenvalue analysis) for the preconditioners is presented in Appendix \ref{Section_Eigenvalues}.

\emph{Remark:} Notice that the preconditioners \eqref{preconditioner} and \eqref{preconditioner_perturbed} are likely to be significantly more sparse than $B_i$, and consequently we expect that these preconditioners will be cost effective to apply in practice. To see this, notice that the blocks $C_i$ are generally much sparser than the linking blocks $D_i$ so that $\p = C_i^TC_i$ is much sparser than $C_i^TC_i+D_i^TD_i$.

\subsubsection{Experiment parameters and results}
\label{S_Numerical_results}
The purpose of this experiment is to study the use of an iterative technique (CG or PCG) to determine the update used at each iteration of the inexact block coordinate descent method, and compare this approach with Exact CD. For Exact CD, the system \eqref{normal_eqi} was solved by forming the Cholesky Decomposition of $B_i$ for each $i$ and then performing two triangular solves to find the exact update.

In the first two experiments, simulated data was used to generate $A$ and the solution vector $x_*$. For each matrix $A$, each block $C_i$ has approximately 20 nonzeros per column, and the density of the linking constraints $D_i$ is approximately $0.1\,\ell \,N_i$. The data vector $b$ was generated from $b=Ax_*$, so the optimal value is known in advance: $F^* = 0$. The stopping condition and tolerance $\epsilon$ for ICD are: $F(x_K) -F^* = \tfrac12\|Ax_K-b\|_2^2 < \epsilon = 0.1$.

The inexactness parameters are set to $\alpha = 0$ and $\beta = 0.1$. Therefore, the update for each block is accepted when $\tfrac12\|A_i T_{\delta_k}\ii - r\|_2^2 \leq \beta = \delta_k^{(i)} = 0.1 $ for all $i,k$. Moreover, each block was chosen with uniform probability $\frac1n$ in all experiments in this section.

In the first experiment the blocks $C_i$ are tall. The incomplete Cholesky decomposition of the preconditioner $\p$ was found using {\sc{Matlab}}'s `\texttt{ichol}' function with a drop tolerance set to $0.1$. The results of this experiment are shown in the Table \ref{Experiment_Over} and all results are averages over 20 runs.

In the second experiment the blocks $C_i$ are wide.\footnote{To ensure that $C_i$ has full rank, a multiple of the identity $I_{m_i}$ is added to the first $m_i$ columns of $C_i$.} The incomplete Cholesky decomposition of the perturbed preconditioner $\hp = \p+\rho I$ (with $\rho = 0.5$) was formed was found using {\sc{Matlab}}'s `\texttt{ichol}' function with a drop tolerance set to $0.1$. The results are shown in the Table \ref{Experiment_Under} and all results are averages over 20 runs.

We briefly explain the terminology used in the tables presented in this section. `Time' represents the cpu time in seconds. Further, the term `block updates' refers to the total number of block updates computed throughout the algorithm; dividing this number by $n$ gives the number of `epochs', which is (approximately) equivalent to the total number of full dimensional matrix-vector products required by the algorithm. The abbreviation `o.o.m.' is the out of memory token.

\begin{table}[h!]\centering
\caption{Results of Exact CD, ICD with CG and ICD with PCG on a quadratic objective with block angular structure using simulated data. For all of these problems, the blocks $C_i$ are tall, and the preconditioner \eqref{preconditioner} is used for ICD with PCG. The size of $A$ ranges from $10^6 \times 10^5$ to $10^7 \times 10^6$. All results are averages over 20 runs. }
\begin{tabular}{|l  |c|c|c||c|c||c|c|c||c|c|c|}
     \hline
      \multicolumn{4}{|c||}{} & \multicolumn{2}{|c||}{Exact CD}& \multicolumn{3}{|c||}{ICD with CG} & \multicolumn{3}{|c|}{ICD with PCG}\\
      \hline
$n$ & $M_i$ & $N_i$ &$\ell$ & \begin{minipage}{1.4cm}\strut
   Block Updates
 \end{minipage} & Time &  \begin{minipage}{1.4cm}\strut
   Block Updates
 \end{minipage} & \begin{minipage}{1.4cm}\strut
   CG\\ Iterations\\[-2.3ex]
 \end{minipage}&Time &\begin{minipage}{1.4cm}\strut
   Block Updates
 \end{minipage} & \begin{minipage}{1.4cm}\strut
   PCG\\ Iterations\\[-2.3ex]
 \end{minipage}&Time \\
\hline
\hline
 100 & $10^4$ & $10^3$ & 1 & 4,820.1 & 37.42 & 4,726.3 & 15,126 & 13.95 & 5,230.6  & 11,379 & 12.59\\
 \hline
 100 & $10^4$ & $10^3$ & 10 & 7,056.7 & 53.94 & 7,181.1 & 14,480 & 17.88 & 6,864.0 & 13,516 & 15.95\\
\hline
100 & $10^4$ & $10^3$ & 100 & 19,129 & 151.97 & 19,411 & 37,841 & 46.32 & 19,446 & 41,344 & 51.12\\
\hline
\hline
10 & $10^5$ & $10^4$ & 1 & 3129.4 & 2488.2 & 3,307.5 & 5,316.4 & 64.39 & 3,246.8 & 4,201.4 & 62.71\\
\hline
10 & $10^5$ & $10^4$ & 10 & 4588 & 3738.6 & 4,753.6 & 9,907.6 & 109.79 & 4,655.4 & 7,646.8 & 104.65\\
\hline
10 & $10^5$ & $10^4$ & 100 & 12,431 & 15,302 & 15,938 & 35,943 & 446.81 & 15,417 & 29,272 & 391.12\\
\hline
\hline
100 & $10^5$ & $10^4$ & 1 & o.o.m. & o.o.m. & 44,799 & 59,340 & 821.64 & 43,427 & 49,801 & 783.11\\
\hline
100 & $10^5$ & $10^4$ & 10 & o.o.m. & o.o.m. & 63,654 & 101,163 & 1,302.0 & 59,351 & 82,097  & 1,267.3\\
\hline
100 & $10^5$ & $10^4$ & 100 & o.o.m. & o.o.m. & 207,314 & 329276 & 4982.8 & 204070 & 302,308 & 4806.1\\
\hline
\end{tabular}
\label{Experiment_Over}
\end{table}

The results presented in Table \ref{Experiment_Over} show that ICD with either CG or PCG  significantly outperforms Exact CD in terms of cpu time. When the blocks are of size $M_i \times N_i = 10^4 \times 10^3$, ICD is approximately 3 times faster than Exact CD. The results are even more striking as the block size increases. Notice that ICD was able to solve problems of all sizes, whereas Exact CD ran out of memory on the problems of size $10^7 \times 10^6$. Further, we notice that PCG is faster than CG in terms of cpu time, demonstrating the benefits of preconditioning. These results strongly support the ICD method.

\begin{table}[h!]\centering
\caption{Results of Exact CD, ICD with CG and ICD with PCG on a quadratic objective with block angular structure using simulated data. For all of these problems, the blocks $C_i$ are wide, and the preconditioner \eqref{preconditioner_perturbed} with $\rho = 0.5$ is used for ICD with PCG. The size of $A$ ranges from $10^5 \times 10^5$ to $10^6 \times 10^6$. All results are averages over 20 runs.}
\begin{tabular}{|c  |c|c|c||c|c||c|c|c||c|c|c|}
     \hline
      \multicolumn{4}{|c||}{} & \multicolumn{2}{|c||}{Exact CD}& \multicolumn{3}{|c||}{ICD with CG} & \multicolumn{3}{|c|}{ICD with PCG}\\
      \hline
$n$ & $M_i$ & $N_i$ &$\ell$ & \begin{minipage}{1.4cm}\strut
   Block Updates
 \end{minipage} & Time &  \begin{minipage}{1.4cm}\strut
   Block Updates
 \end{minipage} & \begin{minipage}{1.4cm}\strut
   CG\\ Iterations\\[-2.3ex]
 \end{minipage}&Time &\begin{minipage}{1.4cm}\strut
   Block Updates
 \end{minipage} & \begin{minipage}{1.4cm}\strut
   PCG\\ Iterations\\[-2.3ex]
 \end{minipage}&Time \\
\hline
\hline
10 & $9,999$ & $10^4$ & 1 & 34.2 & 190.62 & 821.2 & 1957 & 12.55 & 471.4 & 1597 & 9.29\\
\hline
10 & $9,990$ & $10^4$ & $10$ & 31.3 & 191.96 & 1,500.8 & 4,793.3 & 45.81 & 867.7 & 3612 & 24.55\\
\hline
10 & $9,000$ & $10^4$ & $10^3$ & 25.5 & 287.79 & 703.6 & 4,052.8 & 58.31 & 439.0 & 4,309.8 & 46.74\\
\hline
10 & $7,500$ & $10^4$ & $2,500$ & 39.7 & 336.69 & 532.0 & 3183 & 76.77 & 386.5 & 4592 & 70.09\\
\hline
\hline
100 & $9,999$ & $10^4$ & 1 & o.o.m. & o.o.m. & 13077 & 27321 & 185.31 & 8280 & 25715 & 143.63\\
\hline
100 & $9,900$ & $10^4$ & $10^2$ & o.o.m. & o.o.m. & 12,979 & 50,685 & 397.47 & 6,159 & 47,034 & 245.89\\
\hline
100 & $9,000$ & $10^4$ & $10^3$ & o.o.m. & o.o.m. & 6974 & 39535 & 453.18 & 4797 & 52,665 & 496.35\\
\hline
100 & $7,500$ & $10^4$ & $2,500$ & o.o.m. & o.o.m. & 4936 & 28986 & 542.69 & 4246 & 57001 & 740.75\\
\hline
\end{tabular}
\label{Experiment_Under}
\end{table}

The results presented in Table \ref{Experiment_Under} show that ICD outperforms Exact CD. ICD is able to solve all problem instances, whereas Exact CD gives the out of memory token on the large problems. We see that when $\ell$ is small, ICD with PCG has an advantage over ICD with CG. However, when $\ell$ is large, the the preconditioner $\hp$ is not as good an approximation to $A_i^TA_i$ and so ICD with CG is preferable.

\emph{Remark:} Notice that in several of the numerical experiments, Exact CD returned the out of memory token. Exact CD requires the matrices $B_i = C_i^TC_i + D_i^TD_i$ for all $i$ to be formed explicitly, and the Cholesky factors to be found and stored. Even if $A_i$ is sparse, $B_i$ need not be, and the Cholesky factor could be dense, making it very expensive to work with.  Moreover, this problem does not arise for ICD with CG (and arises to a much lesser extent for PCG) because $B_i$ is never explictly formed. Instead, only sparse matrix vector products: $B_ix \equiv C_i^T(C_i x) + D_i^T(D_i x)$ are required. This is why ICD performs extremely well, even when the blocks are very large.

\subsubsection{Real-world data}

In the third experiment we test ICD on a quadratic objective with block angular structure, where the matrices arise from real-world applications. In particular, we have taken several matrices from the Florida Sparse Matrix Collection \cite{Davis11} that have block angular structure. The matrices used are given in Table \ref{Florida_matrices_specifications}. Note that in each case we have taken the transpose of the original matrix to ensure that the matrix is tall. Further, in each case the upper block (recall \eqref{A_matrix}) is diagonal, so we have scaled each of the matrices so that $C = I$. Note that in this case $\p = I$ so there is no need for preconditioning. We compare Exact CD with ICD using CG. All the stopping conditions and algorithm parameters are the same as those given in Section \ref{S_Numerical_results}.
\begin{table}[h!]\centering
\caption{Block angular matrices from the Florida Sparse Matrix Collection \cite{Davis11}. (\texttt{cep1} is from the Meszaros Group while all others are from the Mittelmann Group.) Note that the dimensions given in the table are for the \emph{transpose} of the original test matrix.}
\begin{tabular}{|l |c |c |c|}
     \hline
& $M$ & $N$ & $\ell$\\
\hline
\texttt{cep1} & 4769 & 1521 & 3,248 \\
\texttt{neos} & 515,905 & 479,119 & 36,786\\
\texttt{neos1} & 133,473 & 131,528 & 1,945\\
\texttt{neos2} & 134,128 & 132,568 & 1,560\\
\texttt{neos3} & 518,832 & 512,209 & 6,623\\
\hline
\end{tabular}
\label{Florida_matrices_specifications}
\end{table}

The results of the numerical experiments on these matrices are shown in Table \ref{Experiment_UFSMC}. (To determine $n$ (the number of blocks) and $N_i$ the size of the blocks, we have simply taken the prime factorization of $N$.) ICD with CG performs extremely well on these test problems. In most cases ICD with CG needs more iterations than Exact CD to converge, yet ICD requires only a fraction of the cpu time needed by Exact CD.

\begin{table}[h!]\centering
\caption{Results showing the performance of Exact CD and ICD with CG applied to a quadratic function with the block angular matrices described in Table \ref{Florida_matrices_specifications}. For the small problem \texttt{cep1}, Exact CD is the best algorithm. For all other matrices, ICD with CG is significantly better than Exact CD in terms of the cpu time. }
\begin{tabular}{|l  |c|c||c|c||c|c|c|}
     \hline
      \multicolumn{3}{|c||}{} & \multicolumn{2}{|c||}{Exact CD}& \multicolumn{3}{|c|}{ICD with CG} \\
      \cline{2-8}
 & $n$ & $N_i$ & \begin{minipage}{1.4cm}\strut
   Block Updates
 \end{minipage} & Time &  \begin{minipage}{1.4cm}\strut
   Block Updates
 \end{minipage} & \begin{minipage}{1.4cm}\strut
   CG\\ Iterations\\[-2.3ex]
 \end{minipage}&Time\\
\hline
\hline
\multirow{2}{*}{\texttt{cep1}} &  9 & 169  & 446 & 0.18 & 448  &828 &0.61 \\
\cline{2-8}
& 3 & 507 & 376 & 0.29 & 342 & 678 & 0.52 \\
\hline
\texttt{neos}  & 283 & 1,693 & 622,659  & 3,258.8 & 869,924 & 3,919,172 & 2,734.65\\
\hline
\multirow{2}{*}{\texttt{neos1}}  & 41 & 3,208  & 148,228 & 8,759.6 &  143,156 & 592,070 & 773.70\\
\cline{2-8}
& 8 & 16,441 & 25,503 & 52,113 & 25,853 & 116,468 & 446.26\\
\hline
\multirow{2}{*}{\texttt{neos2}}  & 73 & 1,816 & 329,749 & 4,669.1 & 439,296 & 1,825,835 & 997.04\\
\cline{2-8}
& 8 & 16,571 & 82,784 & 11,518 & 55,414 & 255,129 &  972.27\\
\hline
\texttt{neos3}  & 107 & 4,787 & 81,956 & 9,032.1 & 82,629 & 433,354 & 700.82\\
\hline
\end{tabular}
\label{Experiment_UFSMC}
\end{table}

\subsection{A numerical experiment for a nonsmooth objective}

In this numerical experiment we consider the $l_1$-regularized least squares problem
\begin{equation}
\label{Problem_CS}
  \min_{x\in \R^N} \frac{1}{2} \|Ax - b\|_2^2 + \lambda \|x\|_1,
\end{equation}
where $A \in \R^{M \times N}$, $b \in \R^M$ and $\lambda >0$. Problem \eqref{Problem_CS} fits into the framework \eqref{F} with $f = \frac{1}{2} \|Ax - b\|_2^2$ and $\Psi = \lambda \|x\|_1 = \lambda \sum_{i=1}^n \| x\ii \|_1$. For this experiment we set $B_i = A_i^TA_i$ and $l_i = 1$ for $i = 1,\dots,n$. (It can be shown that this choice of $B_i$ and $l_i$ satisfy the overapproximation \eqref{S2_upperbound}.) Further, for this experiment we use uniform probabilities, $p_i = \frac1n$ for all $i$, and we set $\alpha = 0$ and $\beta > 0$. The algorithm stopping condition is $F(x_k)-F^*< \epsilon = 10^{-4}$, (the data was constructed so that $F^*$ is known), and the regularization parameter was set to $\lambda= 0.01$.

The exact update for the $i$th block is computed via
\begin{equation}
\label{Eqn_l1_subprob}
  V_i(x_k,t) = \langle A_i^Tr_k, t\rangle  + \frac{1}{2}t^TA_i^TA_i t + \lambda \|x_k\ii + t\|_1 = \frac12 \|A_it + r_k\|_2^2 + \lambda \|x_k\ii + t\|_1,
\end{equation}
where $r_k \eqdef Ax_k - b$ and $\nabla_i f(x) = A_i^Tr_k$. Notice that \eqref{Eqn_l1_subprob} \emph{does not have a closed form solution}, meaning that only an \emph{inexact update} can be used in this case. Recall that the inexact update must satisfy \eqref{Td_def}, and for \eqref{Eqn_l1_subprob}, we do not know the optimal value $V_i(x_k,T_0\ii)$. In this case, to ensure that \eqref{Td_def} is satisfied, we simply find the inexact update $T_{\delta_k}\ii$ using an algorithm that terminates on the duality gap. That is, we accept $T_{\delta_k}\ii$ using a stopping condition of the same form as that given by \eqref{StoppingCondition_Dual}.

In the numerical experiments presented in this section, we use the BCGP algorithm \cite{Tappenden11} to solve for the update at each iteration of ICD. This is a gradient based method that solves problems of the form \eqref{Eqn_l1_subprob}, and terminates on the duality gap.

We conduct two numerical experiments. In the first experiment $A$ is of size $0.5N \times N$ where $N = 10^5$. In this case \eqref{Problem_CS} is convex (but not strongly convex.) This means that the complexity result of Theorem \ref{Thm:C-N} applies. In the second experiment $A$ is of size $2N \times N$ where $N = 10^5$. In this case \eqref{Problem_CS} is strongly convex, and the complexity result of Theorem \ref{Thm_ICD_strong} apply.

The purpose of these experiments is to investigate the effect of different levels of inexactness (different values of $\beta$) on the algorithm runtime. In particular we used three different values: $\beta \in \{10^{-4},10^{-6},10^{-8}\}$. To make this a fair test, for each problem instance, the block ordering was fixed in advance. (i.e., before the algorithm begins we form and store a vector whose $k$th element is a index between 1 and $n$ that has been chosen with uniform probability, corresponding to the block to be updated at iteration $k$ of ICD.) Then, ICD was run three times using this block ordering, once for each value of $\beta \in \{10^{-4},10^{-6},10^{-8}\}$. In all cases we use $\delta_k\ii = \beta$ for all $i$ and $k$.

\begin{figure}[h!]\centering
  \includegraphics[width=7cm]{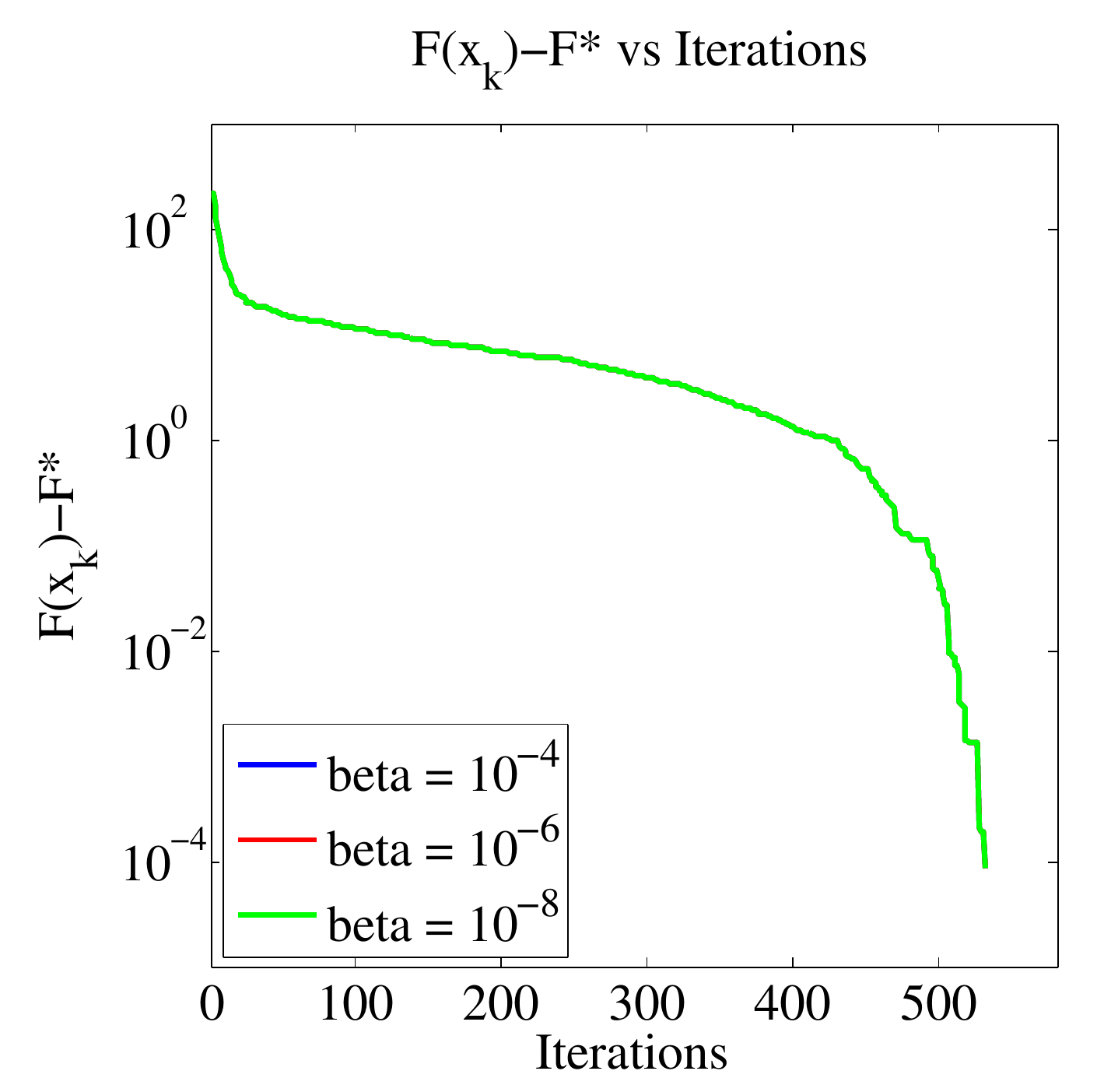}
\includegraphics[width=7cm]{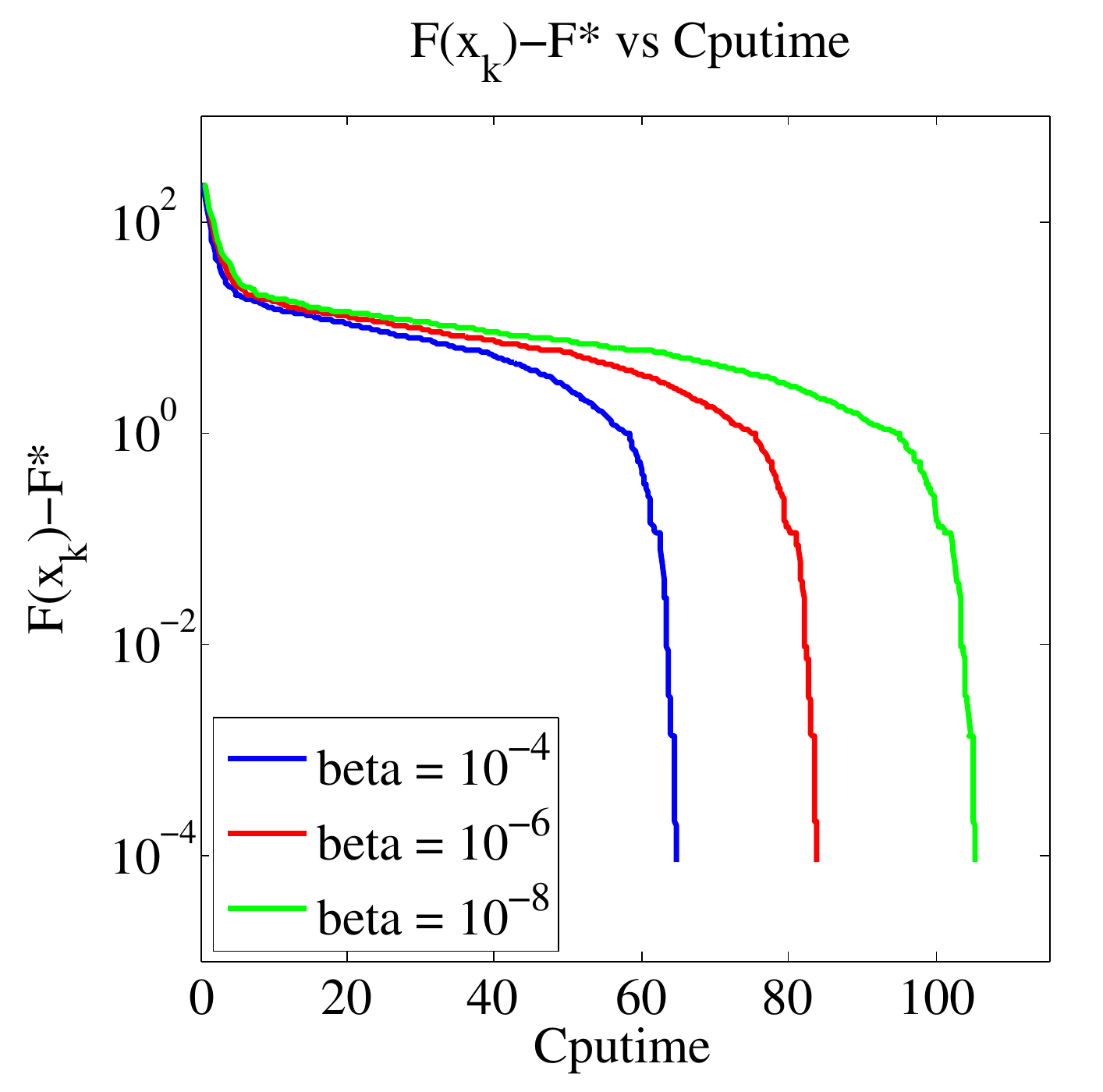}\\
\includegraphics[width=7cm]{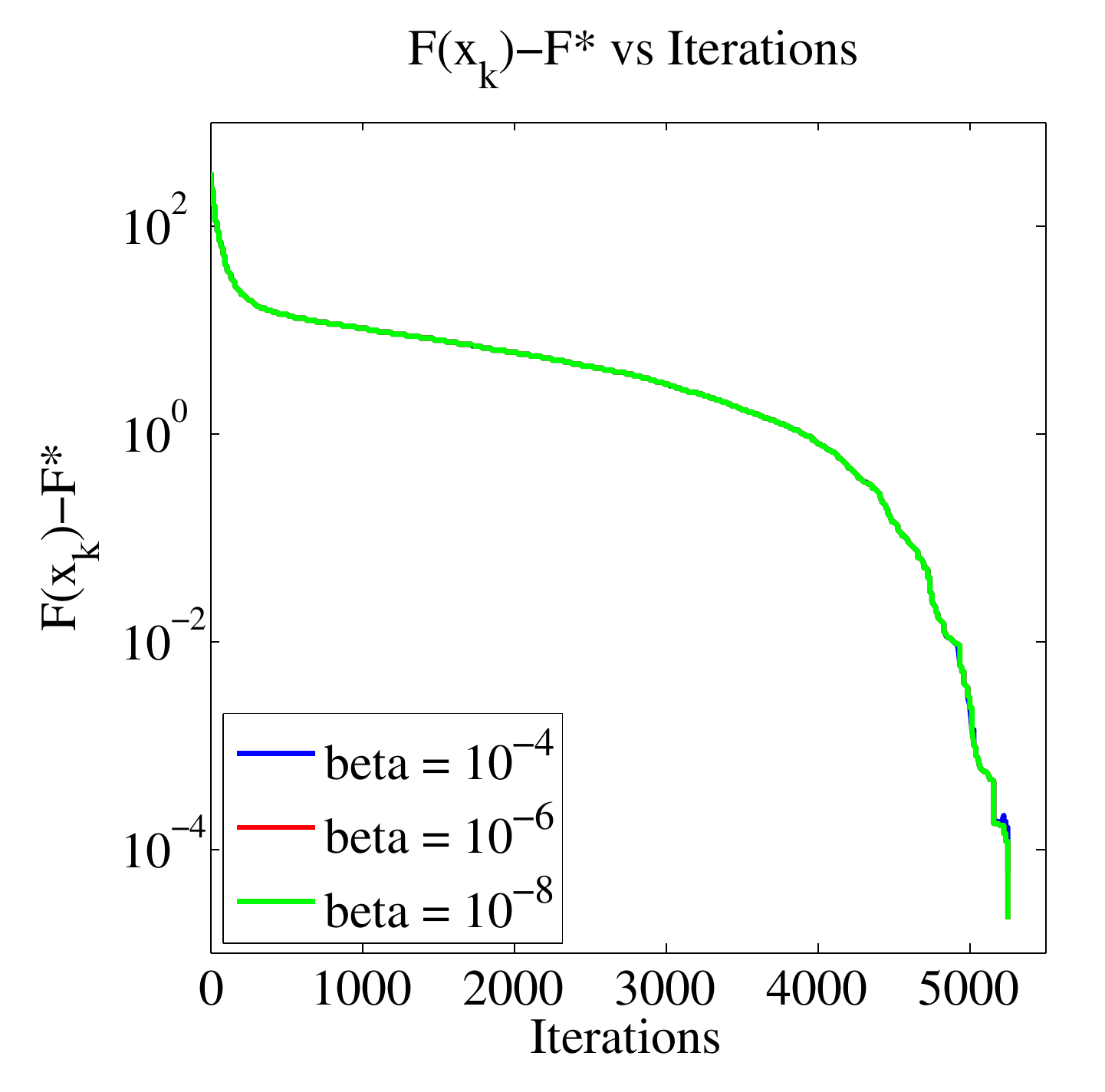}
\includegraphics[width=7.5cm]{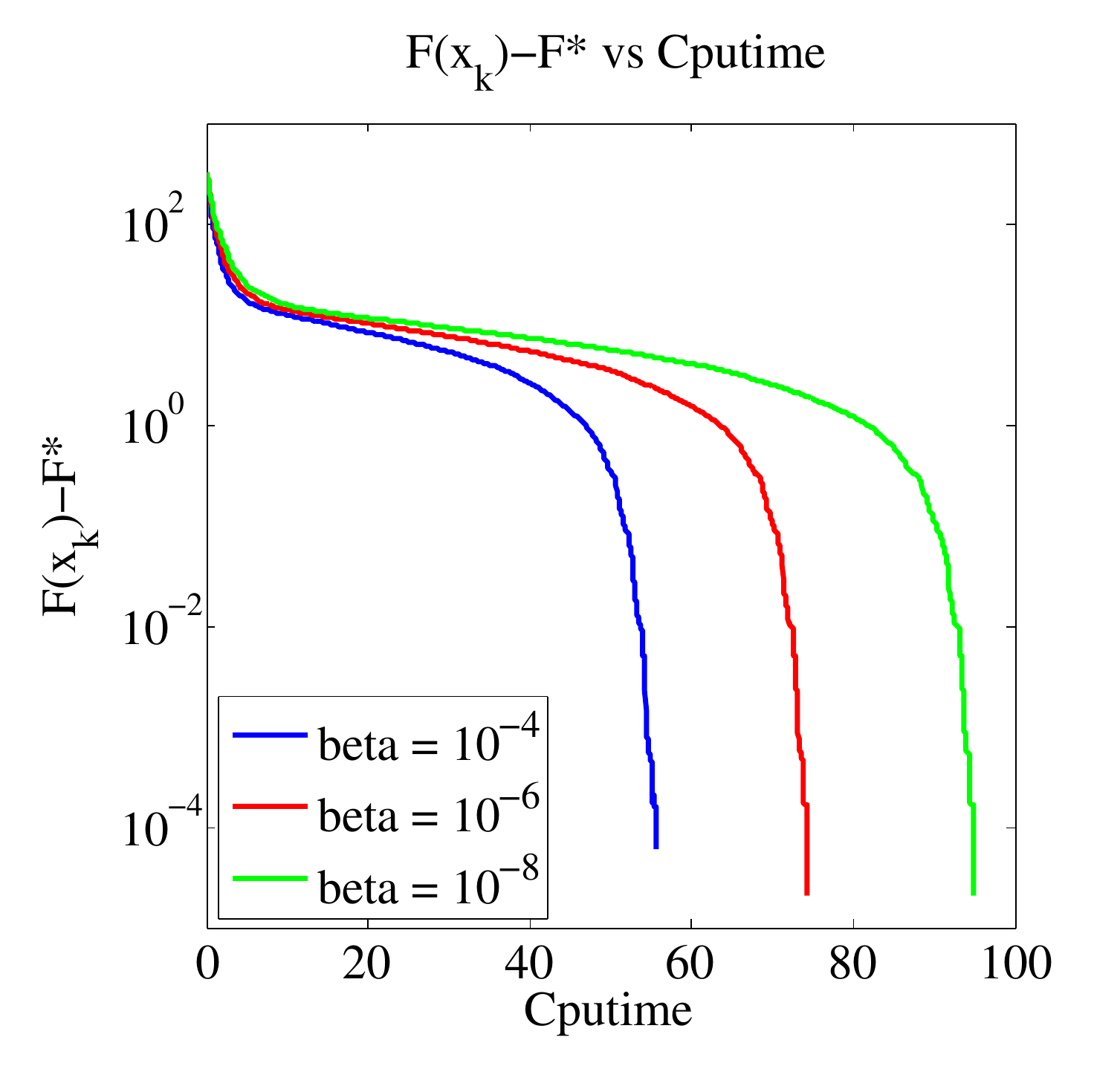}\\
\caption{Plots of the objective function value vs the number of iterations, and the cputime vs the objective function values on the $l_1$-regularised quadratic loss problem \eqref{Problem_CS}. For these plots, the matrix $A$ is $0.5N \times N$, where $N = 10^5$. In the plots in the first row, $n = 10$ and $N_i = 10^4$ for all $i=1,\dots,n$. In the second row, $n = 100$ and $N_i = 10^3$ for all $i=1,\dots,n$.}
\label{Figure_05N}
\end{figure}

\begin{figure}[h!]\centering
  \includegraphics[width=7cm]{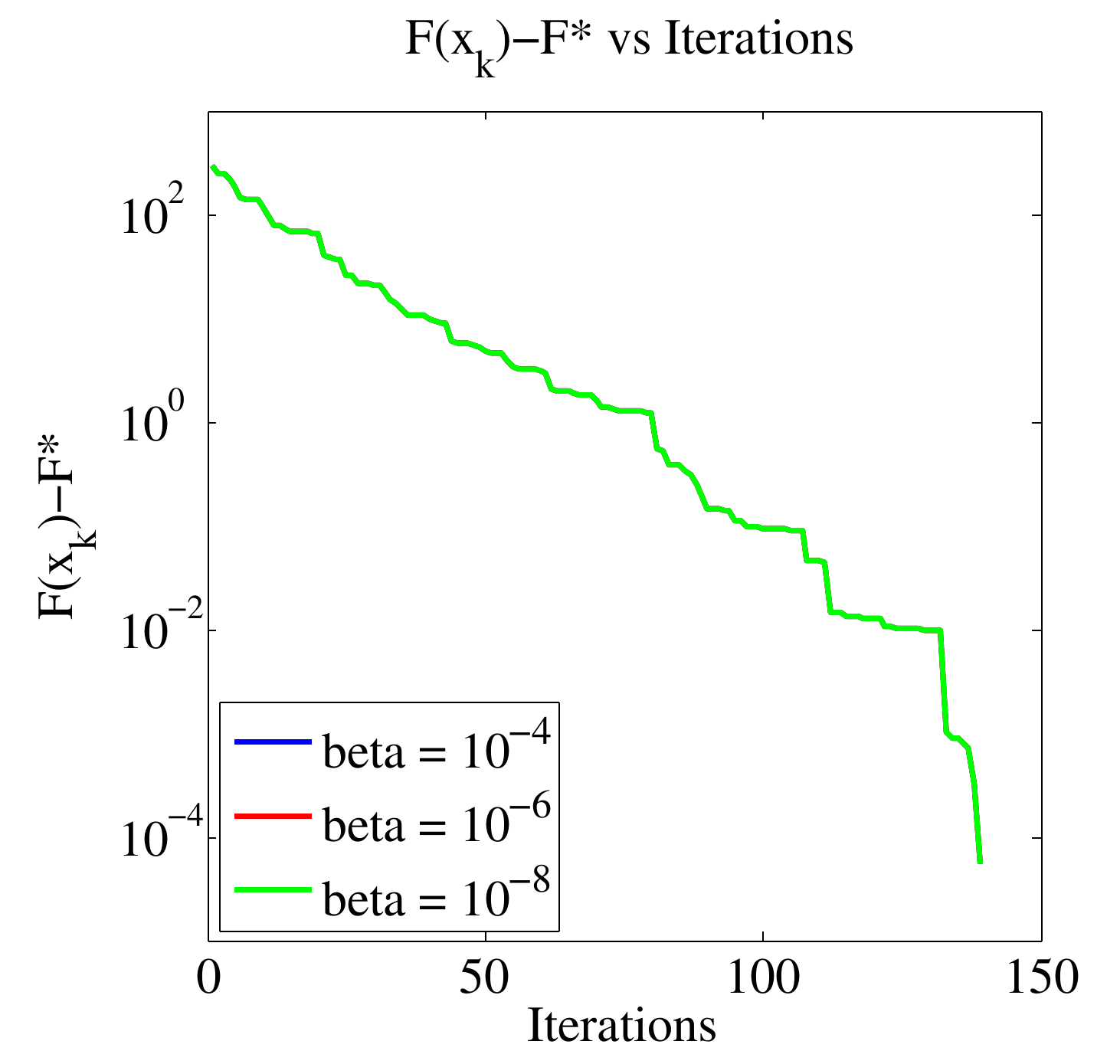}
\includegraphics[width=7cm]{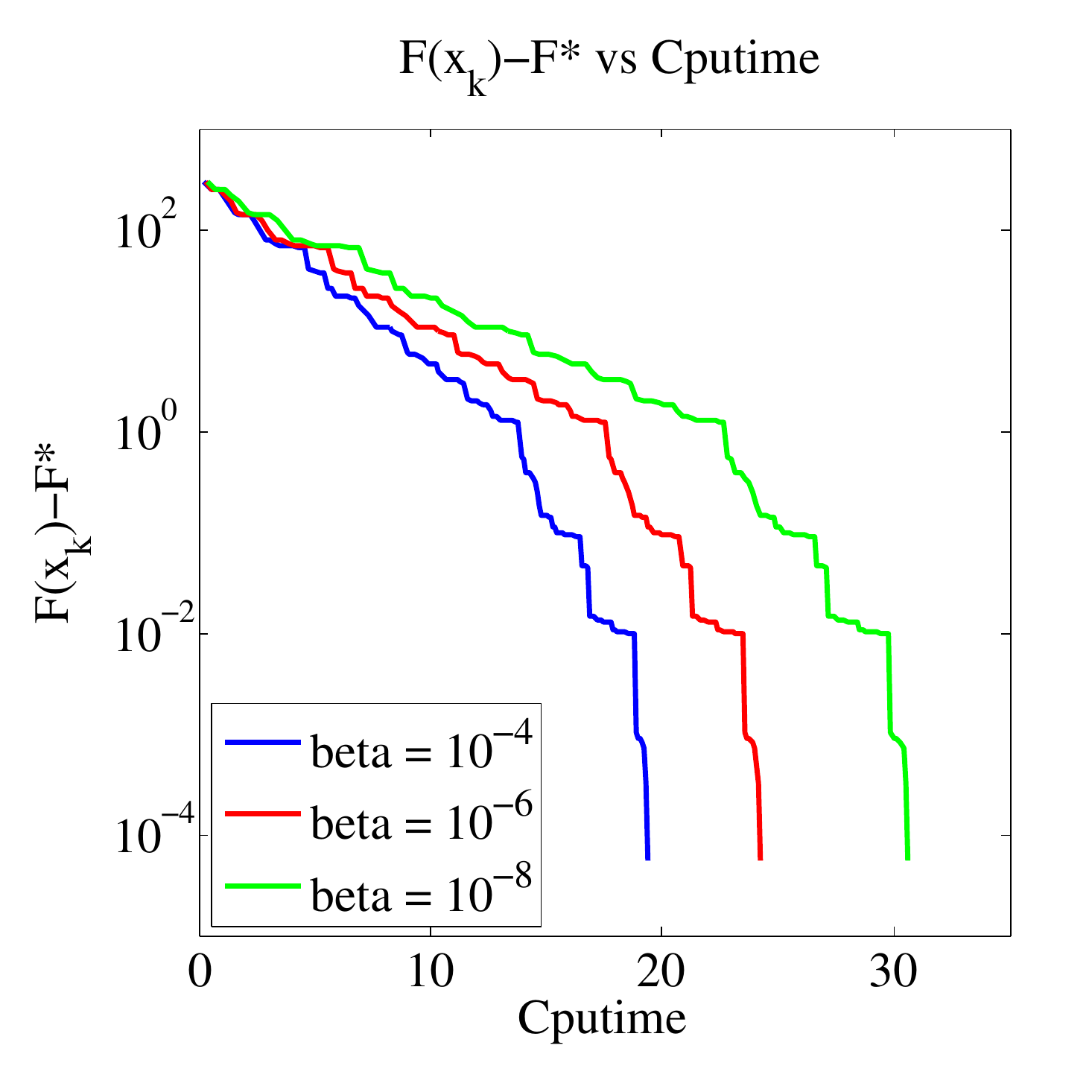}\\
\includegraphics[width=7cm]{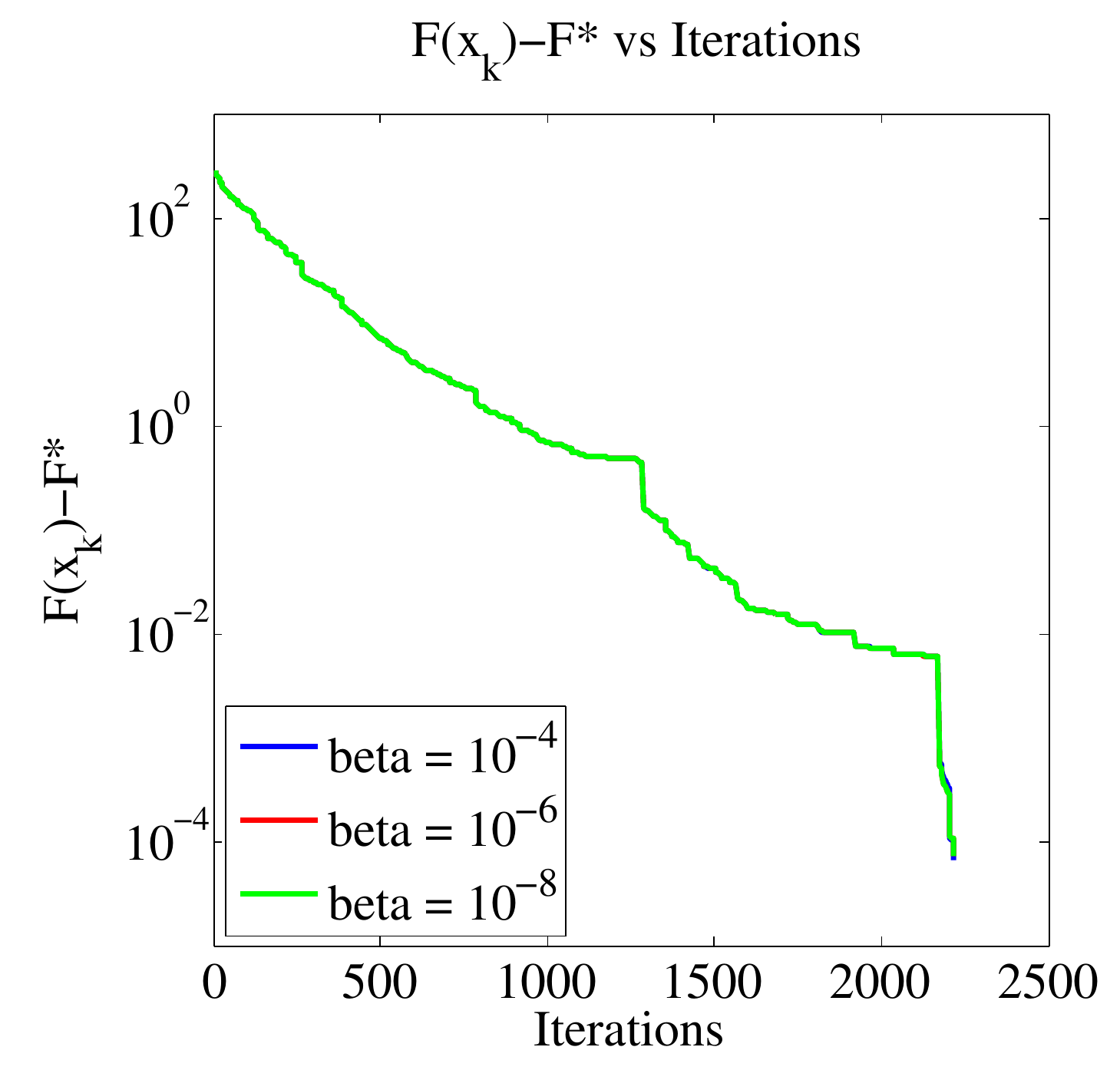}
\includegraphics[width=7cm]{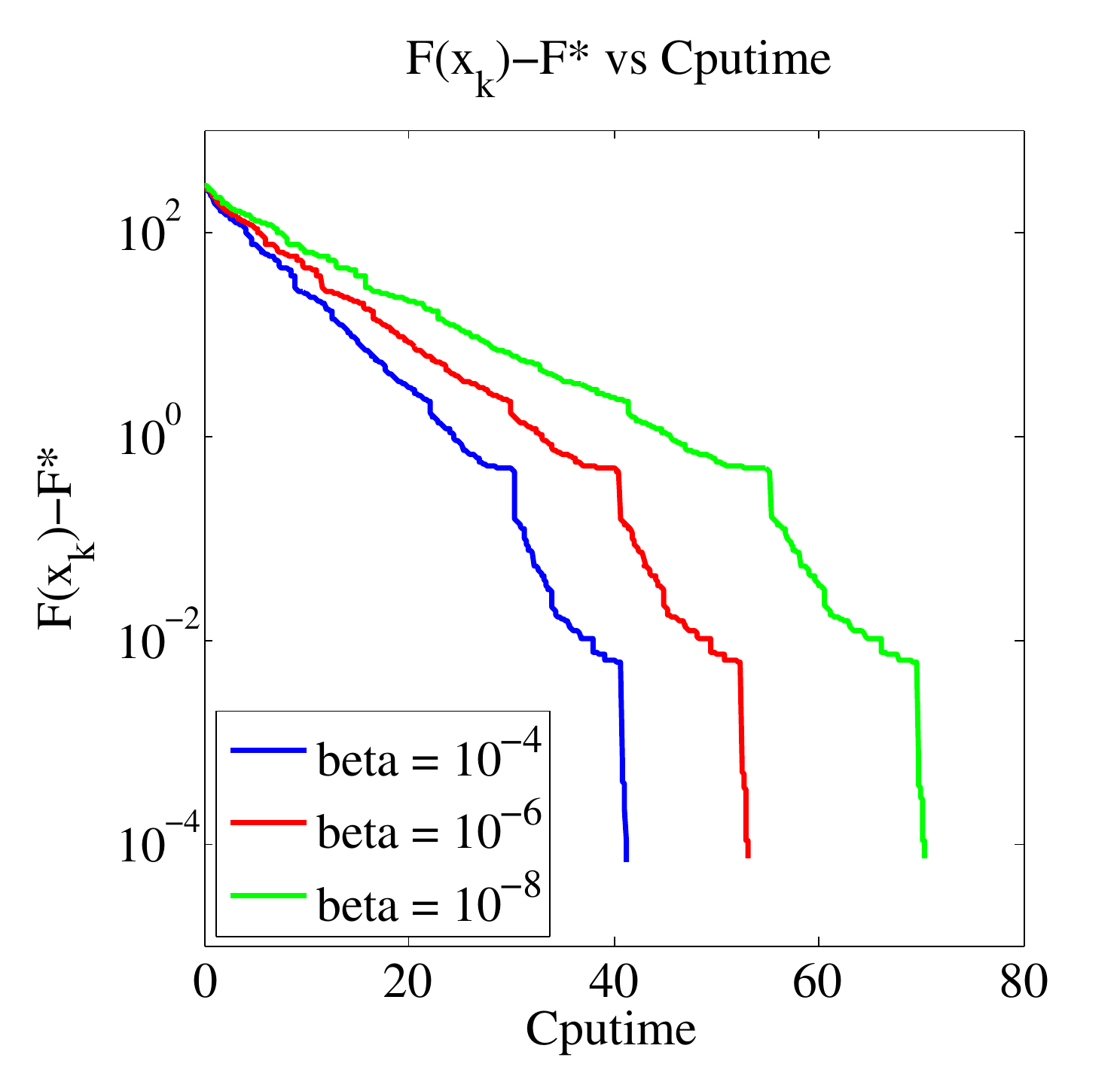}\\
\caption{Plots of the objective function value vs the number of iterations, and the cputime vs the objective function values on the $l_1$-regularised quadratic loss problem \eqref{Problem_CS}. For these plots, the matrix $A$ is $2N \times N$, where $N = 10^5$. In the plots in the first row, $n = 10$ and $N_i = 10^4$ for all $i=1,\dots,n$. In the second row, $n = 100$ and $N_i = 10^3$ for all $i=1,\dots,n$.}
\label{Figure_2N}
\end{figure}

Figures \ref{Figure_05N} and \ref{Figure_2N} show the results of experiments 1 ($M < N$) and 2 ($M > N$) respectively. The experiments were performed many times on simulated data and the plots shown are a particular instance that is \emph{representative of the typical behaviour} observed using ICD on this problem description. We see that when the same block ordering is used, all algorithms essentially require the same number of iterations until termination regardless of the parameter $\beta$. (This is to be expected.) Moreover, it is clear that using a smaller value of $\beta$, corresponding to more `inexactness' in the computed update, leads to a reduction in the algorithm running time, without affecting the ultimate convergence of ICD. This shows that using an inexact update (an iterative method) has significant practical advantages.

%
%
%

\bibliographystyle{plain}
\bibliography{My_references}

\appendix
\section{Eigenvalues of the preconditioned matrix}
\label{Section_Eigenvalues}

The purpose of this section is to provide a full theoretical justification for the choice of preconditioner presented in Section \ref{Section_Numerical_problemdescription}.

The convergence speed of many iterative methods, such as CG, depends on the spectrum of the system matrix. The purpose of a preconditioner is to shift the spectrum of the preconditioned matrix so that the eigenvalues of the resulting system are clustered around one, with few outliers. In this section we study the eigenvalues of the preconditioned matrix under the problem setup described in Sections \ref{Section_quadratic} and \ref{Section_Numerical_problemdescription}.

 Applying \eqref{preconditioner} and \eqref{preconditioner_perturbed} to $B_i(= C_i^TC_i + D_i^TD_i)$ gives
\begin{eqnarray}
\label{Preconditioned}
     \p^{-1}B_i = I + \p^{-1} \Di^T\Di,\text{ and } \hp^{-1}B_i = \hp^{-1}\p + \hp^{-1} \Di^T\Di
\end{eqnarray}

To investigate the quality of a preconditioner, we study the eigenvalues of the preconditioned matrices $\p^{-1}B_i$ and $\hp^{-1}B_i$ defined in \eqref{Preconditioned}.

The nonzero eigenvalues of the $N_i \times N_i$ matrix $\p^{-1}\Di^T\Di$ are the same as the nonzero eigenvalues of $\Di\p^{-1}\Di^T$ (See Theorem 2.8 in \cite{Zhang99}). We prefer to work with $\Di\p^{-1}\Di^T$ because it is symmetric and positive semidefinite, so it has real, nonnegative eigenvalues. (Furthermore, if $\Di$ has full (row) rank, then  $\Di\p^{-1}\Di^T$ is positive definite.)

We can say more about the eigenvalues of the preconditioned matrix by considering the blocks of $A$ and investigating the relationship between the matrices $C$ and $D$, defined in \eqref{C}. Recall that $C$ contains blocks $C_{i} \in \R^{M_i \times N_i}$. The remainder of this section is broken into two parts. The first part considers the case when $M_i \geq N_i$ while the second part considers the case when $M_i < N_i$. In each case $C_i$ is assumed to have full rank.\footnote{Note that the eigenvalues of $\p^{-1}\Di^T\Di$ can be determined exactly by solving the generalized eigenvalue problem $\Di^T\Di v = \lambda \p v$.}

\subsection{Tall blocks}

In this section it is assumed that $C_{i} \in \R^{M_i \times N_i}$ where $M_i \geq N_i$ and that $\Di \in \mathbb{R}^{\ell \times N_i}$ with $1 \leq \ell < N_i$. Furthermore, we assume that $C_i$ has full column rank (the rows of $\Ci$ contain a basis for $\mathbb{R}^{N_i}$) so each row of $\Di$ is a linear combination of the rows of $\Ci$. i.e., for $Z_i \in\R^{\ell \times M_i}$ we can write
\begin{eqnarray}
\label{Dmatrix2}
     \Di = Z_i C_i.
\end{eqnarray}
We have the following result.
\begin{theorem}
\label{eigsM}
Let $r_i = {\rm{rank}}(\Di)$ and $r_i \leq N_i$. Then $\p^{-1}B_i = I+\p^{-1}\Di^T\Di$ \eqref{Preconditioned} has
\begin{itemize}
     \item[(i)] $r_i$ eigenvalues that are strictly greater than one.
     \item[(ii)] $N_i-r_i$ eigenvalues equal to one.
\end{itemize}
\end{theorem}
We have the following bound on the diagonal elements of $\p^{-1} \Di^T\Di$.

\begin{lemma}
     Let $\p \in \R^{N_i \times N_i}$ and $\Di \in \R^{\ell \times N_i}$ be the matrices defined in \eqref{preconditioner} and \eqref{Dmatrix2} respectively and let $z_j^T$ denote the $j$th row of $Z_i$. Let $C_i = Y_i R_i$ denote the thin QR factorization (\cite{Golub96}) of $C_i$, so $Y_i \in \R^{M_i \times N_i}$ has orthonormal columns and $R_i \in \R^{N_i \times N_i}$ is upper triangular. Then
     \begin{eqnarray}
\label{Skyscraper_result}
	  {\rm{trace}}(\Di\p^{-1} \Di^T) \;=\; \sum_{j=1}^\ell \|z_j^TY_i\|_2^2 \;\leq\; \|Z_i\|_F^2.
     \end{eqnarray}
\end{lemma}
\begin{proof}
   The trace is simply the sum of the diagonal entries of a (square) matrix, so:
\begin{eqnarray*}
     D_i\p^{-1}D_i = Z_iC_i(C_i^TC_i)^{-1}C_i^TZ_i^T = Z_iY_i R_i(R_i^TY_i^TY_i R_i)^{-1}R_i^TY_i^TZ_i^T = (Z_iY_i)(Z_iY_i)^T.
\end{eqnarray*}
Now $(\Di\p^{-1} \Di^T)_{jj} = \|Y_i^Tz_j\|_2^2$ and so $\|Z_i\|_F^2 = \sum_{j=1}^{\ell}\|z_j\|_2^2$. Because $Y_iY_i^T$ is a projection matrix, $\|Y_i^Tz_j\|_2^2 = \|Y_iY_i^Tz_j\|_2^2 \leq \|z_j\|_2^2$, and the result follows.
\end{proof}

\emph{Remark:} When $C_i$ is square and has full rank, $Y_i$ is an orthogonal matrix, and subsequently ${\rm{trace}}(\Di\p^{-1} \Di^T) = \sum_{j=1}^\ell \|z_j\|_2^2 = \|Z_i\|_F^2$.

We now present the main result of this section.
\begin{theorem}
     Suppose that $A \in \R^{M \times N}$ has primal block angular structure, with rectangular blocks $\Ci \in \R^{M_i \times N_i}$ ($M_i \geq N_i$) of full rank ($N_i = \text{rank }(\Ci)$) along the diagonal. Suppose that $B_i\equiv A_i^TA_i$, $\Di$ and $\p$ are defined in \eqref{C}, \eqref{Dmatrix2} and \eqref{preconditioner} respectively, and let $r_i = {\rm{rank}}(\Di)$. Then $\p^{-1}B_i$ has
\begin{itemize}
     \item[(i)] $N_i - r_i$ eigenvalues equal to one,
     \item[(ii)] $r_i$ eigenvalues that are strictly greater than 1, and sum to $r_i + \sum_{j=1}^\ell \|Y_i^Tz_j\|_2^2$.
\end{itemize}
\end{theorem}
\begin{proof}
  The proof follows from Theorem \ref{eigsM} and Lemma \ref{Skyscraper_result}.
\end{proof}
\subsection{Wide blocks}

In this section we assume that $C_{i} \in \R^{M_i \times N_i}$ where $M_i < N_i$ with full row rank $M_i = {\text{rank}}(C_i)$, and that $\Di \in \R^{\ell \times N_i}$ where $\ell \geq N_i - M_i$. Then the rows of $\Ci$ form a basis for a subspace $\mathcal{W} \eqdef {\rm{span}}\{c^{(i)}_1,\dots,c^{(i)}_{M_i}\} \subset \R^{N_i}$, where $(c^{(i)}_{j})^T$ is the $j$th row of $\Ci$.

Furthermore, let $W$ be an $\ell \times N_i$ matrix whose rows $w_j^T \in \mathcal{W}$, for $j = 1,\dots,\ell$, and let $W^{\perp}$ be an $\ell \times N_i$ matrix whose rows $(w_j^{\perp})^T \in \mathcal{W}^{\perp}$, for $j = 1,\dots,\ell$, where $\mathcal{W}^{\perp}$ denotes the orthogonal complement of $\mathcal{W}$. Then one can write
\begin{equation}
\label{Warehouse_Di}
     D_i = W + W^{\perp}.
\end{equation}

For ICD, $B_i = \Ai^T\Ai$ must have full rank because it defines a norm (see Section \ref{S_Block_structure}). However, when $C_i$ is wide, $\p$ defined in \eqref{preconditioner} is rank deficient so we use the preconditioner $\hp$ defined in \eqref{preconditioner_perturbed}. The preconditioned matrix is defined in \eqref{Preconditioned}. We study the eigenvalues of $\hp^{-1} \p$ and $\hp^{-1} \Di^T\Di$ separately, before stating the main result of this section, which describes the eigenvalues of $\hp^{-1}B_i$.

\begin{theorem}
     \label{thm_eigs_PP}
Let $C_i$ be a real $M_i \times N_i$ matrix with $M_i < N_i$ and full row rank $M_i = {\rm{rank}}(C_i)$. Let $\p$ and $\hp$ be defined in \eqref{preconditioner} and \eqref{preconditioner_perturbed} respectively, and let $M_i = {\rm{rank }}(\p)$. Then $\hp^{-1} \p$ has $N_i-M_i$ zero eigenvalues and $M_i$ positive eigenvalues that tend to 1 as $\rho \to 0$.
\end{theorem}
\begin{proof}
The preconditioner $\hp$ satisfies ${\rm{rank}}(\hp^{-1}\p) = {\rm{rank}}(\p) = M_i$, so $\hp^{-1} \p$ has $M_i$ nonzero eigenvalues and $N_i-M_i$ zero eigenvalues. Furthermore, the $M_i$ nonzero eigenvalues are positive. (Notice that $\hp^{-\frac{1}{2}}\p \hp^{-\frac{1}{2}}$ is positive semidefinite.)

Let  $\lambda_1, \dots, \lambda_{M_i}$ denote the $M_i$ nonzero eigenvalues of $\p$. The eigenvalue decomposition of $\p$ is $\p = V \Lambda V^T$, where $\Lambda = \text{diag}(\lambda_1, \dots, \lambda_{M_i},0,\dots,0)$. Moreover, the eigenvalue decomposition for $\hp$ is $\hp = \p + \rho I = V \hat{\Lambda} V^T$ where $\hat{\Lambda} = \Lambda + \rho I$.

Finally, $\hp^{-1}\p = V \hat{\Lambda}^{-1}V^TV \Lambda V^T = V \hat{\Lambda}^{-1} \Lambda V^T$, where $\hat{\Lambda}^{-1} \Lambda$ is a diagonal matrix with diagonal entries $(\lambda_1/(\lambda_1+\rho),\dots,\lambda_{M_i}/(\lambda_{M_i}+\rho),0,\dots,0)$ and as $\rho \to 0$, $\frac{\lambda_j}{\lambda_j+ \rho} \to 1$ for $j = 1,\dots,M_i$.
\end{proof}

\begin{theorem}
\label{Thm_warehouse_D}
Let $D_i$ and $\hp$ be as defined in \eqref{Warehouse_Di} and \eqref{preconditioner_perturbed} respectively, and let $\rho>0$. Then
  \begin{equation}
    {\rm{tr}}(D_i\hp^{-1}D_i^T) = \sum_{j=1}^\ell \; \Big(\|\hat{\Lambda}_1^{-\frac{1}{2}} V_1^T(w_j + w_j^{\perp})\|_2^2 + \frac{1}{\rho}\|V_2^T(w_j^{\perp})\|_2^2.
  \end{equation}
\end{theorem}
\begin{proof}
Recall that $\hp = V \hat{\Lambda} V^T$ where $\hat{\Lambda} = \Lambda + \rho I$ and let $V = \begin{bmatrix} V_1 & V_2 \end{bmatrix}$ be a partitioning of $V$, where $V_1 \in R^{N_i \times M_i}$, and $V_2 \in R^{N_i \times (N_i-M_i)}$. The columns of $V_2$ form a basis for the null space of $\Ci$, so $V_2^Tw = 0$.
The results follows by expanding $(\Di\hp^{-1}\Di^T)_{jj} = (w_{j} + w_{j}^{\perp})^T\hp^{-1}(w_{j} + w_{j}^{\perp})$.

\end{proof}
Theorems \ref{thm_eigs_PP} and \ref{Thm_warehouse_D} demonstrate the importance of the parameter $\rho$. A small value of $\rho$ will lead to a good clustering of the eigenvalues around one, but if $\rho$ is too small then $(w_j^{\perp})^T\hp^{-1}(w_j^{\perp})$ will become arbitrarily large. Hence, there is a trade-off here.

Now we state the main result of this section, which gives bounds on the eigenvalues of $\hp^{-1} \Di^TD_i$.
\begin{theorem}
\label{Thm_warehouse}
     Let $C_i$ be an $M_i\times N_i$ matrix with $M_i < N_i$ and $M_i = {\rm{rank}}(C_i)$ and let $D_i$ be an $\ell \times N_i$ matrix with $r_i = {\rm{rank}}(D_i)$. Let $\hp$ be the preconditioner defined in \eqref{preconditioner_perturbed} and let $A_i$ be defined in \eqref{Ai_matrix} with $N_i \geq s_i = {\rm{rank}}(A_i)$ and $B_i = A_i^TA_i$. Then $\mathcal{M}_i = \hp^{-1}B_i$ has
\begin{itemize}
\item[(i)] $N_i - s_i$ eigenvalues equal to zero.
     \item[(ii)] $s_i-r_i$ eigenvalues in the interval $(0,1)$
\item[(iii)] $r_i$ eigenvalues in the interval 
     $\Big(1,1 +  \sum_{j=1}^\ell \; \Big(\|\hat{\Lambda}_1^{-\frac{1}{2}} V_1^T(w_j + w_j^{\perp})\|_2^2 + \frac{1}{\rho}\|V_2^T(w_j^{\perp})\|_2^2\Big) \Big)$
\end{itemize}
\end{theorem}
\begin{proof}
     Part (i) holds because $B_i$ is $N_i \times N_i$ with ${\rm{rank}}(B_i) = {\rm{rank}}(A_i) = s_i$.  Part (ii) follows from Theorem \ref{thm_eigs_PP} and Theorem \ref{thm_eigs_PP}. For part (iii), notice that $\lambda_{\max}(\Di \hp^{-1}\Di^T) \leq {\rm{trace}}(\Di \hp^{-1} \Di^T)$. Using \eqref{Thm_warehouse_D} and \cite[Theorem 4.3.1]{Horn85}, gives the result.
\end{proof}

\emph{Remark:} For ICD we require $\text{rank}(A_i) = N_i$, because this ensures that $B_i$ is a positive definite matrix. Notice that in this case, Theorem \ref{Thm_warehouse} shows that all eigenvalues of $\mathcal{M}_i= \hp^{-1}B_i$ are strictly greater than zero (i.e., $N_i = s_i$).

\end{document}